\documentclass[reqno]{amsart}
\usepackage{hyperref}

\usepackage{a4wide,calrsfs}
\usepackage{amsmath,bbm}
\usepackage{color}
\usepackage{epsfig}
\usepackage{subfigure}
\usepackage{amssymb}
\usepackage{enumerate}
\usepackage{graphicx}
\usepackage{subfigure}
\usepackage{amsmath}
\usepackage{rotating}
\usepackage{stmaryrd}
\usepackage{upgreek}
\usepackage{cite}

\begin{document}
\title[A Nonlinear Third Order Differential and Applications]
{L$^p$-Solutions of a Nonlinear Third Order Differential Equation and
Asymptotic Behavior of Linear Fourth Order
Differential Equations}

\author[A. Coronel, F. Huancas, M. Pinto ]
{An{\'\i}bal Coronel, Fernando Huancas, Manuel Pinto}  

\address{An{\'\i}bal Coronel, \newline
GMA, Departamento de Ciencias B\'asicas,\newline
Facultad de Ciencias, Universidad del B\'{\i}o-B\'{\i}o,\newline
Campus Fernando May, Chill\'{a}n, Chile}
\email{acoronel@ubiobio.cl}

\address{Fernando Huancas, \newline
GMA, Departamento de Ciencias B\'asicas,\newline
Facultad de Ciencias, Universidad del B\'{\i}o-B\'{\i}o,\newline
Campus Fernando May, Chill\'{a}n, Chile}
\email{fihuanca@gmail.com}

\address{Manuel Pinto,\newline
Departamento de Matem\'atica,\newline
Facultad de Ciencias, Universidad de Chile,\newline
Santiago, Chile}
\email{pintoj@uchile.cl}

\thanks{\today}

\begin{abstract}
In this paper we prove the well-posedness and we study the asymptotic behavior
of nonoscillatory $L^p$-solutions for a third order nonlinear scalar  differential
equation. The equation consists of two parts: a linear third order with constant
coefficients part and a nonlinear part represented by a polynomial
of fourth order in three variables with variable coefficients. The results
are obtained assuming three hypotheses: (i) the characteristic polynomial
associated with the linear part has simple and real roots, (ii) the coefficients
of the polynomial satisfy asymptotic integral smallness conditions, and
(iii) the polynomial coefficients are in $L^p([t_0,\infty[)$. 
These results are applied to study a fourth order linear differential equation of Poincar\'e
type and a fourth order linear differential equation with unbounded coefficients.
Moreover, we give some examples where the classical theorems can not be applied.
\end{abstract}

\maketitle
\numberwithin{equation}{section}
\newtheorem{theorem}{Theorem}[section]
\newtheorem{proposition}[theorem]{Proposition}
\newtheorem{definition}{Definition}[section]
\newtheorem{lemma}{Lemma}[section]
\allowdisplaybreaks

\section{Introduction}

The study of asymptotic
behavior for linear ordinary differential equations has been 
achieved their breakthroughs in the twentieth century,
see \cite{coronel_ejqde_2015} for a short historical review 
of the main landmarks, see also
\cite{bellman_book,codilevi_book,coppel_book,eastham_book,
fedoryuk_book}. In \cite{coronel_ejqde_2015}
is revised the contributions of Poincar\'e \cite{poincare_1885},
Perron \cite{perron_1909}, Levinson \cite{levinson1948}, 
Hartman-Wintner \cite{hartman_wintner_1955} and Harris and Lutz
\cite{harris_lutz_1977,harris_lutz_1974}. Now,
it is undisputed in the recent years the increasing interest in the
development of the theory for the asymptotic analysis
of high order linear differential equations, see for instance
\cite{coronel_ejqde_2015,figueroa_2008,figueroa_2010,figueroa_2015,
barbara_2013,stepin_2005,stepin_2010}. The main common motivation of 
these works come from two sources: the real world applications
and the deepening of the theory. In particular, in the 
scalar method there is an interesting problem of the second kind,
where the analysis
of the asymptotic behavior of a linear equation is obtained via 
the analysis of a particular nonlinear equation.
Thus, a natural question is to study the general form of 
that nonlinear equation.

In this paper we are interested in the study
of $L^p$-solutions and the asymptotic behavior
of the following third order nonlinear differential equation
\begin{subequations}
 \label{eq:general_no_lin}
\begin{eqnarray}
 z'''(t)+ b_2 z''(t)+b_1 z'(t)+b_0 z(t)=\mathbb{P}(t,z(t),z'(t),z''(t)),
 \label{eq:general_no_lin:1}
\end{eqnarray}
where  $b_i$ are real constants and $\mathbb{P}:\mathbb{R}^4\to\mathbb{R}$ is a given function
such that   $\mathbb{P}(x_0,x_1,x_2,x_3)$ is a polynomial of fourth degree in the 
three variables $(x_1,x_2,x_3)$ and its coefficients depends on $x_0$, i.e.
 $\mathbb{P}$ admits the representation 
\begin{eqnarray}
\mathbb{P}(x_0,\mathbf{x})=\sum_{|\boldsymbol{\alpha}|=0}^{4}
\Omega_{\boldsymbol{\alpha}}(x_0)\mathbf{x}^{\boldsymbol{\alpha}},
\quad
\mbox{with}
\quad
\mathbf{x}=(x_1,x_2,x_3),
\quad
\boldsymbol{\alpha}=(\alpha_1,\alpha_2,\alpha_3)\in\{0,1,2,3,4\}^3.
 \label{eq:general_no_lin:2}
\end{eqnarray}
\end{subequations}
Here, we have used the notation $\{0,1,2,3,4\}^3$ for the cartesian product
and $|\boldsymbol{\alpha}|$ and $\mathbf{x}^{\boldsymbol{\alpha}}$
for the standard multindex notation, i.e. 
$|\boldsymbol{\alpha}|=\alpha_1+\alpha_2+\alpha_3$ and 
$\mathbf{x}^{\boldsymbol{\alpha}}=x_1^{\alpha_1}x_2^{\alpha_2}x_3^{\alpha_3}$.

The analysis of \eqref{eq:general_no_lin} is mainly motivated by the application 
of the scalar method to study the asymptotic behavior of nonoscillatory 
solutions for a fourth order linear differential equations of Poincar\'e type:
\begin{eqnarray}
y^{({\rm iv})}
+[a_3+r_3 (t)] y'''
+[a_2+r_2 (t)] y''
+[a_1+r_1 (t)] y'
+[a_0+r_0 (t)] y=0,
\label{eq:intro_uno_intro}
\end{eqnarray}
where $a_i$ are constants and 
$r_i$ are real-valued functions.
Indeed, we recall that the scalar method consists of three big steps:
(i) a change of variable of the type $z(t)=y'(t)[y(t)]^{-1}-\mu$ with
$\mu$ a characteristic root associated to \eqref{eq:intro_uno_intro} when
$r_0=r_1=r_2=r_3=0$,
(ii) the analysis of existence, uniqueness and
asymptotic behavior of the new nonlinear equation in terms of $z$
which is given by 
\begin{align}
&z'''(t)+[4\mu  +a_3]z''(t)
    +[6\mu ^2+3a_3\mu  +a_2]z'(t)
    +[4\mu  ^3+3\mu  ^2a_3+2\mu  a_2
    +a_1]z(t) 
    \nonumber\\
    & \qquad
   =-\Big\{r_3(t)z''(t)+[3\mu  r_3(t)+r_2(t)]z'(t)
   +[3\mu  ^2r_3(t)+2\mu  r_2(t)+r_1(t)]z(t)
   \nonumber\\
   & \qquad\qquad
   +\mu  ^3r_3(t) +\mu  ^2r_2(t)+\mu  r(t)+r_0(t)
   +4z(t)z''(t)+[12\mu    +3a_3+3r_3(t)]z(t)z'(t)
    \nonumber\\
    &\qquad\qquad
   +6z(t)^2z'(t)    
   +3[z'(t)]^2+[6\mu  ^2+3\mu  a_3+a_2
   +3\mu  r_3(t)+r_2(t)]z(t)^2
    \nonumber\\
   & \qquad\qquad  
   +[4\mu  +r_3(t)]z(t)^3+z(t)^4
   \Big\},
   \label{eq:ricati_type_original}
\end{align}
and (iii) the translation of the results from $z$ to the original
variable via the relation $y(t)=\exp\Big(\int_{t_0}^t (z(s)+\mu)ds\Big)$.
Now, we note that the equation
\eqref{eq:ricati_type_original}
is of the type \eqref{eq:general_no_lin}. 

Recently, in \cite{coronel_ejqde_2015} the analysis of \eqref{eq:ricati_type_original} 
in the context of $C^2_0$-solutions was developed
by assuming  three hypotheses. The first hypothesis
is related to the constant coefficients $a_i$ and set that
the characteristic polynomial associated with 
the homogeneous equation for \eqref{eq:ricati_type_original} (i.e. when
$r_0=r_1=r_2=r_3=0$)
has simple and real roots. The other two hypotheses 
are related to the behavior of 
the perturbation functions $r_i$ and establish
asymptotic integral smallness conditions 
of the perturbations.  
Now, under these general hypotheses and by application 
of a fixed point argument the authors prove that \eqref{eq:ricati_type_original} 
has  a unique solution in 
\begin{eqnarray*}
C_0^2([t_0,\infty[)
=\Big\{z\in C^2([t_0,\infty[,\mathbb{R})\quad :\quad
z,z',z''\to 0\text{ when }t\to\infty\Big\},
\quad t_0\in\mathbb{R}
\end{eqnarray*}
and also obtain  the 
asymptotic behavior
of the solutions for \eqref{eq:ricati_type_original}. Moreover,
they  establish the existence 
of a fundamental system of solutions and precise 
the formulas for the asymptotic behavior of 
the linear fourth order  differential equation \eqref{eq:intro_uno_intro}.

In this paper we prove the existence, uniqueness 
and asymptotic behavior for \eqref{eq:general_no_lin}.
Our results improve our previous results obtained 
in \cite{coronel_ejqde_2015} for \eqref{eq:ricati_type_original}, 
since the equation \eqref{eq:ricati_type_original} is a particular
case of the equation \eqref{eq:general_no_lin}. 
More precisely, we deduce the well posedness
in $C^2_0([t_0,\infty[)$ by assuming that
\begin{itemize}
\item  The roots $\gamma_1,\gamma_2$ and $\gamma_2$
of  the corresponding characteristic
polynomial associated to 
$z'''+  b_2 z''+b_1 z'+b_0 z=0$
(the homogeneous part of \eqref{eq:general_no_lin})
are real and simple.

\item The coefficients $\Omega_{\boldsymbol{\alpha}}$
satisfy the following requirements
\begin{eqnarray*}
&&\lim_{t\to\infty}
\left|\int_{t_0}^{\infty}g(t,s)\Omega_{\boldsymbol{\alpha}}(s)ds\right|
	+\left|\int_{t_0}^{\infty}\frac{\partial g}{\partial t}(t,s)
	\Omega_{\boldsymbol{\alpha}}(s)ds\right|
	+\left|\int_{t_0}^{\infty}\frac{\partial^2 g}{\partial t^2}(t,s)
	\Omega_{\boldsymbol{\alpha}}(s)ds\right|=0,
\\
&&
\lim_{t\to\infty}
\int_{t_0}^{\infty}\left[|g(t,s)|
	+\left|\frac{\partial g}{\partial t}(t,s)\right|
	+\left|\frac{\partial^2 g}{\partial t^2}(t,s)\right|
	\right]\Big|\sum_{|\boldsymbol{\alpha}|=1}
	\Omega_{\boldsymbol{\alpha}}(s)\Big|ds=0,
\\
&&\sum_{k=1}^4
\int_{t_0}^{\infty}\left[|g(t,s)|
	+\left|\frac{\partial g}{\partial t}(t,s)\right|
	+\left|\frac{\partial^2 g}{\partial t^2}(t,s)\right|
	\right]\Big|\sum_{|\boldsymbol{\alpha}|=k}
	\Omega_{\boldsymbol{\alpha}}(s)\Big|ds
	\quad
	\mbox{is bounded when $t\to\infty$.}
\end{eqnarray*}
Here $g$ is a Green function.
\end{itemize}
Now, considering that  
\begin{eqnarray*}
\sum_{|\boldsymbol{\alpha}|=1}^4\Big|\Omega_{\boldsymbol{\alpha}}(s)\Big|\le 
\rho
\quad
\mbox{ for }\rho\in \left]0,\frac{1}{\upvarsigma\hat{A}}\right[,
\end{eqnarray*}
with $\sigma$ and  $\hat{A}$
a given numbers in terms of $\gamma_1,\gamma_2,\gamma_3$,
we get that 
\begin{eqnarray*}
z(t),z'(t),z''(t)
=
\left\{
\begin{array}{ll}
\displaystyle
O\Big(\int_{t}^{\infty}e^{-\beta(t-s)}
\Big|\sum_{|\boldsymbol{\alpha}|=0}\Omega_{\boldsymbol{\alpha}}(s)\Big|
ds\Big),
& (\gamma_1,\gamma_2,\gamma_3)\in \mathbb{R}^3_{---},\;\, \beta\in ]\gamma_1,0[,
\\
\displaystyle
O\Big(\int_{t_0}^{\infty}e^{-\beta(t-s)}
\Big|\sum_{|\boldsymbol{\alpha}|=0}\Omega_{\boldsymbol{\alpha}}(s)\Big|
ds\Big),
& (\gamma_1,\gamma_2,\gamma_3)\in \mathbb{R}^3_{+--},\;\, \beta\in ]\gamma_2,0[,
\\
\displaystyle
O\Big(\int_{t_0}^{\infty}e^{-\beta(t-s)}
\Big|\sum_{|\boldsymbol{\alpha}|=0}\Omega_{\boldsymbol{\alpha}}(s)\Big|
ds\Big),
& (\gamma_1,\gamma_2,\gamma_3)\in \mathbb{R}^3_{++-},\;\, \beta\in ]\gamma_3,0[,
\\
\displaystyle
O\Big(\int_{t_0}^{t}e^{-\beta(t-s)}
\Big|\sum_{|\boldsymbol{\alpha}|=0}\Omega_{\boldsymbol{\alpha}}(s)\Big|
ds\Big),
& (\gamma_1,\gamma_2,\gamma_3)\in \mathbb{R}^3_{+++},\;\, \beta\in ]0,\gamma_3[,
\end{array}
\right.
\;\;\;{}
\end{eqnarray*}
which represents the asymptotic behavior of the solutions
for \eqref{eq:general_no_lin}. 
Moreover, assuming that 
\begin{itemize}
 \item
The coefficients $\Omega_{\boldsymbol{\alpha}}$ of $\mathbb{P}$
are such that $\Omega_{\boldsymbol{\alpha}}\in L^p([t_0,\infty[)$
for $|\boldsymbol{\alpha}|=0$ and  for  $|\boldsymbol{\alpha}|\ge 1$
the functions  $\Omega_{\boldsymbol{\alpha}}$ are of the following type
$
 \Omega_{\boldsymbol{\alpha}}(t)=
 \lambda_{\boldsymbol{\alpha},p}\Omega_{\boldsymbol{\alpha},p}(t)
 +\lambda_{\boldsymbol{\alpha},c},
$
where $\lambda_{\boldsymbol{\alpha},p}$ and $\lambda_{\boldsymbol{\alpha},c}$ are real constants
and $\Omega_{\boldsymbol{\alpha},p}\in L^p([t_0,\infty[)$,
\end{itemize}
we prove that the solution of \eqref{eq:general_no_lin}
is a function belongs of the Sobolev space $W^{2,p}([t_0,\infty[)$
and there exists $m+1$ functions with $m$ depending of $p$ denoted by
  $\Theta_{1},\ldots,\Theta_{m}$ and $\Psi,$  
  such that $\Theta_{k}\in W^{2,p/k}([t_0,\infty[),$
  for $k=1,\ldots,m$, 
  $\Psi\in W^{2,1}([0,\infty[)$ and
  $z^{(\ell)}(t)=\sum_{k=1}^m\Theta^{(\ell)}_{k}(t)+\Psi^{(\ell)}(t)$
  for $\ell=0,1,2.$

On the other side, in this paper
we consider two applications of the results. First,
assuming the hypotheses given in \cite{coronel_ejqde_2015} for $a_i$ and $r_i$ and using 
the fact that perturbation functions $r_i\in L^p([t_0,\infty[)$
we deduce the equation \eqref{eq:ricati_type_original}
has a solution belongs to the Sobolev space $W^{2,p}([t_0,\infty[)$.
Moreover, we obtain a Levison and Hartman-Wintner type results for the
asymptotic behavior of the solutions for \eqref{eq:intro_uno_intro}.
The second application is the study of the 
asymptotic behavior of the solutions for the following
fourth order differential equation
\begin{eqnarray}
y^{({\rm iv})}(t)
-2[q(t)]^{1/2} y'''(t)
-q(t) y''(t)
+2[q(t)]^{3/2} y'(t)
+r(t) y(t)=0,
\label{eq:unbounded_intro}
\end{eqnarray}
where $q$ and $r$ are real valued unbounded functions.
In this case, the
analysis is based in a change of variable, which transforms
\eqref{eq:unbounded_intro} in an equation of 
the type \eqref{eq:intro_uno_intro}.

The paper is organized as follows. 
In section~\ref{sect:analisis_of_general}, we   
develop the analysis of existence, uniqueness and
asymptotic behavior of \eqref{eq:general_no_lin} in  
$C_0^2([t_0,\infty[)$. Now, on section~\eqref{section:lpsolution}
we study the integrability of the
solutions for equation \eqref{eq:general_no_lin} and we obtain a result
for the asymptotic behavior of $L^p$-solutions.
In section~\ref{sec:applications} we present the applications 
of the results to the analysis of \eqref{eq:intro_uno_intro}
and \eqref{eq:unbounded_intro}. Finally, on section~\ref{sec:examples}
we give some examples.

\section{Analysis of equation \eqref{eq:general_no_lin} in $C_0^2([t_0,\infty[)$}
\label{sect:analisis_of_general}

In this section we present the results of well posedness and 
the asymptotic behavior of \eqref{eq:general_no_lin}.

\subsection{Existence and uniqueness of \eqref{eq:general_no_lin}}
Before to present the result of this subsection, we need to define 
some notations about Green functions. First, let us consider 
the equation associated
to \eqref{eq:general_no_lin:1} with $\mathbb{P}=0$, i.e.
\begin{eqnarray}
z'''+  b_2 z''+b_1 z'+b_0 z=0,
\label{eq:riccati_type_equiv_homo}
\end{eqnarray}
and denote by $\gamma_i,\; i=1,2,3,$ the roots of the  characteristic
polynomial for \eqref{eq:riccati_type_equiv_homo}. Then, the Green function for 
\eqref{eq:riccati_type_equiv_homo} is defined by 
\begin{eqnarray}
g(t,s)=\frac{1}{(\gamma_3-\gamma_2)(\gamma_3-\gamma_1)(\gamma_2-\gamma_1)}
\;
\left\{
\begin{array}{lcl}
 g_1(t,s), 
	& \quad & (\gamma_1,\gamma_2,\gamma_3)\in \mathbb{R}^3_{---},
 \\
 g_2(t,s), 
	& & (\gamma_1,\gamma_2,\gamma_3)\in \mathbb{R}^3_{+--},
 \\
 g_3(t,s), 
	& & (\gamma_1,\gamma_2,\gamma_3)\in \mathbb{R}^3_{++-},
 \\
 g_4(t,s), 
	& & (\gamma_1,\gamma_2,\gamma_3)\in \mathbb{R}^3_{+++}, 
\end{array}
\right.
\label{eq:green_function}
\end{eqnarray} 
where 
\begin{eqnarray}
g_1(t,s)&=&
\left\{
\begin{array}{lcl}
 0, & &t\ge s,
 \\ 
(\gamma_2-\gamma_3)e^{-\gamma_1(t-s)}
      +(\gamma_3-\gamma_1)e^{-\gamma_2(t-s)}
      +(\gamma_1-\gamma_2)e^{-\gamma_3(t-s)},
      & \quad & t\le s,
\end{array}
\right.
\qquad\mbox{$ $}
\label{eq:green_function:g1}
\\
g_2(t,s)&=&
\left\{
\begin{array}{lcl}
(\gamma_2-\gamma_3)e^{-\gamma_1(t-s)},
      & &t\ge s,
 \\      
(\gamma_1-\gamma_2)e^{-\gamma_3(t-s)}
      +(\gamma_3-\gamma_1)e^{-\gamma_2(t-s)},
      & \quad & t\le s,
\end{array}
\right.
\label{eq:green_function:g2}
\\
g_3(t,s)&=&
\left\{
\begin{array}{lcl}
(\gamma_2-\gamma_3)e^{-\gamma_1(t-s)}
      +(\gamma_3-\gamma_1)e^{-\gamma_2(t-s)},
      & \quad & t\ge s,
   \\      
(\gamma_2-\gamma_1)e^{-\gamma_3(t-s)},
      & &t\le s,
\end{array}
\right.
\label{eq:green_function:g3}
\\
g_4(t,s)&=&
\left\{
\begin{array}{lcl}
(\gamma_2-\gamma_3)e^{-\gamma_1(t-s)}
      +(\gamma_3-\gamma_1)e^{-\gamma_2(t-s)}
      +(\gamma_1-\gamma_2)e^{-\gamma_3(t-s)},
      & \quad & t\ge s,
 \\
 0, & &t\le s.
\end{array}
\right.
\label{eq:green_function:g4}
\end{eqnarray}
Further details on Green functions may be consulted in~\cite{bellman_book}.
Moreover, given $g$ by \eqref{eq:green_function}, we define the 
functionals  $\mathcal{G}$ and $\mathcal{L}$ as follows
\begin{eqnarray}
  \mathcal{G}(E)(t) 
  &=& \left|\int_{t_0}^{\infty}g(t,s)E(s)ds\right|
	+\left|\int_{t_0}^{\infty}\frac{\partial g}{\partial t}(t,s)E(s)ds\right|
	+\left|\int_{t_0}^{\infty}\frac{\partial^2 g}{\partial t^2}(t,s)E(s)ds\right|,
\label{eq:functyonal_G}
\\
  \mathcal{L}(E)(t) 
  &=& \int_{t_0}^{\infty}\left[|g(t,s)|
	+\left|\frac{\partial g}{\partial t}(t,s)\right|
	+\left|\frac{\partial^2 g}{\partial t^2}(t,s)\right|
	\right]|E(s)|ds.
\label{eq:functyonal_L}
\end{eqnarray}
Note that the inequality $0\le \mathcal{G}(E)(t)\le \mathcal{L}(E)(t)$
holds for all $t\ge t_0.$

\begin{theorem}
\label{thm:general_no_lin}
Let us introduce the notation $C_0^2([t_0,\infty[)$
for the following space of functions
\begin{eqnarray*}
C_0^2([t_0,\infty[)
=\Big\{z\in C^2([t_0,\infty[,\mathbb{R})\quad :\quad
z,z',z''\to 0\text{ when }t\to\infty\Big\},
\quad t_0\in\mathbb{R},
\end{eqnarray*}
and consider the equation \eqref{eq:general_no_lin}
where the constants $b_i$ and the coefficients of $\mathbb{P}$
satisfy the following two restrictions
\begin{enumerate}
\item[($P_1$)]  The roots $\gamma_1,\gamma_2$ and $\gamma_2$
of  the corresponding characteristic
polynomial associated to \eqref{eq:riccati_type_equiv_homo},
the homogeneous part of \eqref{eq:general_no_lin:1},
are real and simple.

\item[($P_2$)] The coefficients $\Omega_{\boldsymbol{\alpha}}$
of $\mathbb{P}$ are such that
\begin{eqnarray*}
&&\mathcal{G}\left(\sum_{|\boldsymbol{\alpha}|=0}
	\Omega_{\boldsymbol{\alpha}}\right)(t)\to 0,
\quad
\mathcal{L}\left(\sum_{|\boldsymbol{\alpha}|=1}
	|\Omega_{\boldsymbol{\alpha}}|\right)(t)\to 0
\quad
\mbox{and}
\\
&&
\sum_{k=1}^4
\mathcal{L}\left(\sum_{|\boldsymbol{\alpha}|=k}^4
	|\Omega_{\boldsymbol{\alpha}}|\right)(t)
\quad
\mbox{bounded when $t\to\infty$.}
\end{eqnarray*}
Here $\mathcal{G}$ and $\mathcal{L}$ are the operators defined on~\eqref{eq:functyonal_G}
and \eqref{eq:functyonal_L}.
\end{enumerate}
hold. Then, there is a unique  $z\in C_0^2([t_0,\infty[)$ solution 
of  \eqref{eq:general_no_lin}.
\end{theorem}

\begin{proof}
By the method of variation of parameters,  the hypothesis $(P_1)$, 
implies that the equation
\eqref{eq:general_no_lin:1} is equivalent to the following integral equation
\begin{eqnarray}
z(t)=\int_{t_0}^\infty g(t,s)\mathbb{P}\Big(s,z(s),z'(s),z''(s)\Big)ds,
\label{eq:integ_equation}
\end{eqnarray}
where $g$ is the Green function defined on  \eqref{eq:green_function}.
Moreover,
we recall that $C_0^2([t_0,\infty[)$ is a Banach space with the norm 
$\|z\|_0=\sup_{t\geq t_0}[|z(t)|+|z'(t)|+|z''(t)|].$
Now, we define the operator $T$ from $C_0^2([t_0,\infty[)$ to 
$ C_0^2([t_0,\infty[)$ as follows
\begin{eqnarray}
Tz(t)&=&\int_{t_0}^\infty g(t,s)\mathbb{P}\Big(s,z(s),z'(s),z''(s)\Big)ds.
\label{eq:operator_fix_point}
\end{eqnarray}
Then, we note that \eqref{eq:integ_equation} can be rewritten as the operator
equation 
\begin{eqnarray}
Tz=z
\qquad
\mbox{over}
\qquad
D_{\upeta}:=\Big\{z\in C_0^2([t_0,\infty[)\quad:\quad \|z\|_0\le\upeta \Big\},
\label{eq:operator_equation}
\end{eqnarray}
where $\upeta\in\mathbb{R}^+$ will be selected
in order to apply the Banach fixed point theorem.
Indeed, we have that

\vskip 0.5cm
\noindent
{\bf (a)} {\it $T$ is well defined from $C_0^2([t_0,\infty[)$ to $C_0^2([t_0,\infty[)$}.
Let us consider an arbitrary $z\in C_0^2([t_0,\infty[).$ We note that
 \begin{eqnarray*}
  T'z(t) &=& 
  \int_{t_0}^{\infty}
	\frac{\partial g}{\partial t}(t,s)
	\mathbb{P}\Big(s,z(s),z'(s),z''(s)\Big)
	ds,
	\\
  T''z(t) &=& 
  \int_{t_0}^{\infty}
	\frac{\partial^2 g}{\partial t^2}(t,s)
	\mathbb{P}\Big(s,z(s),z'(s),z''(s)\Big)
	ds.
\end{eqnarray*}
Then, by the definition of $g$, we immediately deduce that 
$Tz,T'z,T''z\in C^2([t_0,\infty[,\mathbb{R}).$
Furthermore,  by \eqref{eq:general_no_lin:2}, we can deduce the following estimates
\begin{eqnarray}
|z(t)|  &\le & \left|\int_{t_0}^{\infty}
	g(t,s)
	\sum_{|\boldsymbol{\alpha}|=0}
	\Omega_{\boldsymbol{\alpha}}(s)ds \right|
	+\int_{t_0}^{\infty}
	|g(t,s)|
	\sum_{|\boldsymbol{\alpha}|=1}
	\left|\Omega_{\boldsymbol{\alpha}}(s)\right|
	|z(s)|^{\alpha_1}|z'(s)|^{\alpha_2}|z''(s)|^{\alpha_3}
	ds
	\nonumber\\
	&&
	\qquad
	+\int_{t_0}^{\infty}
	|g(t,s)|
	\sum_{|\boldsymbol{\alpha}|=2}^4
	\left|\Omega_{\boldsymbol{\alpha}}(s)\right|
	|z(s)|^{\alpha_1}|z'(s)|^{\alpha_2}|z''(s)|^{\alpha_3}
	ds,
	\label{eq:bound_for_tz0_0} 
\\
  |z'(t)| &\le & \left|\int_{t_0}^{\infty}
	\frac{\partial g}{\partial t}(t,s)
	\sum_{|\boldsymbol{\alpha}|=0}
	\Omega_{\boldsymbol{\alpha}}(s)ds \right|
	+\int_{t_0}^{\infty}
	\left|\frac{\partial g}{\partial t}(t,s)\right|
	\sum_{|\boldsymbol{\alpha}|=1}
	\left|\Omega_{\boldsymbol{\alpha}}(s)\right|
	|z(s)|^{\alpha_1}|z'(s)|^{\alpha_2}|z''(s)|^{\alpha_3}
	ds
	\nonumber\\
	&&
	\qquad
	+\int_{t_0}^{\infty}
	\left|\frac{\partial g}{\partial t}(t,s)\right|
	\sum_{|\boldsymbol{\alpha}|=2}^4
	\left|\Omega_{\boldsymbol{\alpha}}(s)\right|
	|z(s)|^{\alpha_1}|z'(s)|^{\alpha_2}|z''(s)|^{\alpha_3}
	ds,
	\label{eq:bound_for_tz0}
\\
 |z''(t)| &\le & \left|\int_{t_0}^{\infty}
	\frac{\partial^2 g}{\partial t^2}(t,s)
	\sum_{|\boldsymbol{\alpha}|=0}
	\Omega_{\boldsymbol{\alpha}}(s)ds \right|
	+\int_{t_0}^{\infty}
	\left|\frac{\partial^2 g}{\partial t^2}(t,s)\right|
	\sum_{|\boldsymbol{\alpha}|=1}
	\left|\Omega_{\boldsymbol{\alpha}}(s)\right|
	|z(s)|^{\alpha_1}|z'(s)|^{\alpha_2}|z''(s)|^{\alpha_3}
	ds
	\nonumber\\
	&&
	\qquad
	+\int_{t_0}^{\infty}
	\left|\frac{\partial^2 g}{\partial t^2}(t,s)\right|
	\sum_{|\boldsymbol{\alpha}|=2}^4
	\left|\Omega_{\boldsymbol{\alpha}}(s)\right|
	|z(s)|^{\alpha_1}|z'(s)|^{\alpha_2}|z''(s)|^{\alpha_3}
	ds.
	\label{eq:bound_for_tz0_3}
\end{eqnarray}
Now, by application of the hypothesis $(P_2)$, we have that
the right hand sides of 
\eqref{eq:bound_for_tz0_0}-\eqref{eq:bound_for_tz0_3} tend to $0$
when $t\to\infty$. Then,
$Tz,T'z,T''z\to 0$ when $t\to\infty$
or equivalently $Tz\in C^2_0$ for all $z\in C^2_0.$

\vskip 0.5cm
\noindent
{\bf (b)} {\it For all $\upeta\in ]0,1[$, the set $D_{\upeta}$ 
is invariant under $T$}. Let us consider $z\in D_{\upeta}$.
From \eqref{eq:bound_for_tz0_0}-\eqref{eq:bound_for_tz0_3}, 
we get the following estimate
\begin{eqnarray}
\|Tz\|_0
&\le& \mathcal{G}\left(\sum_{|\boldsymbol{\alpha}|=0}
	\Omega_{\boldsymbol{\alpha}}\right)(t)
	+\sum_{k=1}^4\|z\|^k_0 
	\mathcal{L}\left(\sum_{|\boldsymbol{\alpha}|=k}
	|\Omega_{\boldsymbol{\alpha}}|\right)(t).
\label{eq:bound_for_tz0_inv}
\end{eqnarray}
Now, by $(P_2)$ we deduce that 
the first term on the right hand side of \eqref{eq:bound_for_tz0_inv} 
converges to $0$ when $t\to\infty$. Similarly, by application
of $(P_2)$, we can prove that the inequality
\begin{eqnarray*}
&& \sum_{k=1}^4\|z\|^k_0 
	\mathcal{L}\left(\sum_{|\boldsymbol{\alpha}|=k}
	|\Omega_{\boldsymbol{\alpha}}|\right)(t)
\\
&&
\qquad\qquad
\le
\upeta^2
\Bigg\{
\mathcal{L}\left(\sum_{|\boldsymbol{\alpha}|=2}
	|\Omega_{\boldsymbol{\alpha}}|\right)(t)
+
\upeta
\mathcal{L}\left(\sum_{|\boldsymbol{\alpha}|=3}
	|\Omega_{\boldsymbol{\alpha}}|\right)(t)
+
\upeta^2
\mathcal{L}\left(\sum_{|\boldsymbol{\alpha}|=4}
	|\Omega_{\boldsymbol{\alpha}}|\right)(t)
\Bigg\}
\\
&&
\qquad\qquad
\le
\upeta
\end{eqnarray*}
holds
when $t\to\infty$ in a right neighborhood of $\upeta=0$. 
Hence, by \eqref{eq:bound_for_tz0_inv} 
and $(P_2)$,
we prove that $Tz\in D_{\upeta}$ for all $z\in D_{\upeta}$.

\vskip 0.5cm
\noindent
{\bf (c)} {\it $T$ is a contraction for $\upeta\in ]0,1[$}. 
Let $z_1,z_2\in D_{\upeta}$, by by \eqref{eq:general_no_lin:2}
and algebraic rearrangements,
we follow that
\begin{eqnarray*}
\|T z_1-T z_2\|_0
&\leq& 
\sum_{k=1}^4
\|z_1-z_2\|^k_0
\mathcal{L}\left(\sum_{|\boldsymbol{\alpha}|=k}
|\Omega_{\boldsymbol{\alpha}}|\right)(t)
\\
&\le&
\|z_1-z_2\|_0
\max\Big\{1,\upeta,\upeta^2,\upeta^3\Big\}
\sum_{k=1}^4
\mathcal{L}\left(\sum_{|\boldsymbol{\alpha}|=k}
|\Omega_{\boldsymbol{\alpha}}|\right)(t).
\end{eqnarray*}
Then, by application of $(P_2)$, we deduce that $T$ is a contraction,
since, for an arbitrary $\upeta\in ]0,1[$, 
we have that $\max\Big\{ 1,\upeta,\upeta^2,\upeta^3\Big\}=\upeta<1$.

Hence, from (a)-(c) and application of Banach fixed point theorem,
we deduce that there is a unique $z\in D_\upeta\subset
C_0^2([t_0,\infty[)$ solution of \eqref{eq:operator_equation}. 
\end{proof}

\subsection{Asymptotic behavior of the solution for \eqref{eq:general_no_lin}}

Before to presente the result of this subsection, we deduce a 
useful bound of the Green function~\eqref{eq:green_function}. 
To fix indeas, we consider $g_1$ and we have that
\begin{eqnarray}
&&
\left|g_1(t,s)\right|
+
\left|\frac{\partial g_1}{\partial t}(t,s)\right|
+
\left|\frac{\partial^2 g_1}{\partial t^2}(t,s)\right|
\nonumber\\
&&
\hspace{1.5cm}
\leq
 \Big(|\gamma_3-\gamma_2|
	+|\gamma_1-\gamma_3|
	+|\gamma_2-\gamma_1|
	\Big)
e^{-\max\Big\{\gamma_1,\gamma_2,\gamma_3\Big\}(t-s)},
\quad{}
\nonumber\\
&&\hspace{1.7cm}
+
 \Big(|\gamma_3-\gamma_2||\gamma_1|
	+|\gamma_1-\gamma_3||\gamma_2|
	+|\gamma_2-\gamma_1||\gamma_3|	
	\Big)
e^{-\max\Big\{\gamma_1,\gamma_2,\gamma_3\Big\}(t-s)},
\quad{}
\nonumber\\
&&\hspace{1.7cm}
+
 \Big(|\gamma_3-\gamma_2||\gamma_1|^2
	+|\gamma_1-\gamma_3||\gamma_2|^2
	+|\gamma_2-\gamma_1||\gamma_3|^2	
	\Big)
e^{-\max\Big\{\gamma_1,\gamma_2,\gamma_3\Big\}(t-s)}
\nonumber\\
&&\hspace{1.5cm}
\le A e^{-\gamma_1(t-s)},
\label{eq:bound:g1}
\end{eqnarray}
where 
\begin{eqnarray}
 A=|\gamma_3-\gamma_2|(1+|\gamma_1|+|\gamma_1|^2)
 +|\gamma_3-\gamma_1|(1+|\gamma_2|+|\gamma_2|^2)
 +|\gamma_2-\gamma_1|(1+|\gamma_3|+|\gamma_3|^2).
 \label{eq:bound:not:A}
\end{eqnarray}
Analogously, we can prove  the bounds for $i=2,3,$ and in general
we obtain that
\begin{eqnarray}
&&
\left|g_i(t,s)\right|
+
\left|\frac{\partial g_i}{\partial t}(t,s)\right|
+
\left|\frac{\partial^2 g_i}{\partial t^2}(t,s)\right|
\le A e^{-\gamma_i(t-s)}, 
\quad i=1,2,3.
\label{eq:bound:gi}
\end{eqnarray}

\begin{theorem}
\label{teo:solution_asymptotic}
Consider that the hypotheses of Theorem~\ref{thm:general_no_lin}
are satisfied and assume that $\gamma_1>\gamma_2>\gamma_3$. 
Moreover consider the positive number
$\upvarsigma$ depending of $\gamma_1,\gamma_2,\gamma_3$ and a given
number $\beta$ defined as follows
\begin{eqnarray}
\upvarsigma
=
\left\{
\begin{array}{lll}
\displaystyle
\frac{1}{-\gamma_1+\beta},
&\;
& (\gamma_1,\gamma_2,\gamma_3)\in \mathbb{R}^3_{---}\quad\mbox{and}\quad \beta\in ]\gamma_1,0[,
\\
\displaystyle
\frac{1}{-(-\gamma_1+\beta)}+\frac{1}{-\gamma_2+\beta},
&
& (\gamma_1,\gamma_2,\gamma_3)\in \mathbb{R}^3_{+--}\quad\mbox{and}\quad \beta\in ]\gamma_2,0[,
\\
\displaystyle
\frac{1}{-(-\gamma_2+\beta)}+\frac{1}{-\gamma_3+\beta},
&
& (\gamma_1,\gamma_2,\gamma_3)\in \mathbb{R}^3_{++-}\quad\mbox{and}\quad \beta\in ]\gamma_3,0[,
\\
\displaystyle
\frac{1}{-(-\gamma_3+\beta)},
&
& (\gamma_1,\gamma_2,\gamma_3)\in \mathbb{R}^3_{+++}\quad\mbox{and}\quad \beta\in ]0,\gamma_3[.
\end{array}
\right.
\label{eq:asymptotic_teo:sigma}
\end{eqnarray}
If the coefficients of $\mathbb{P}$ satisfies the following estimate
\begin{eqnarray}
\sum_{|\boldsymbol{\alpha}|=1}^4\Big|\Omega_{\boldsymbol{\alpha}}(s)\Big|\le 
\rho
\quad
\mbox{ for }\rho\in \left]0,\frac{1}{\upvarsigma\hat{A}}\right[
\quad
\mbox{ with}
\quad
\hat{A}=\frac{A}{|(\gamma_2-\gamma_1)(\gamma_3-\gamma_2)(\gamma_3-\gamma_1)|},
\label{eq:hipotesis3}
\end{eqnarray}
where  $A$ defined on \eqref{eq:bound:not:A},
then the solution  of \eqref{eq:general_no_lin}
has the following asymptotic behavior 
\begin{eqnarray}
z(t),z'(t),z''(t)
=
\left\{
\begin{array}{ll}
\displaystyle
O\Big(\int_{t}^{\infty}e^{-\beta(t-s)}
\Big|\sum_{|\boldsymbol{\alpha}|=0}\Omega_{\boldsymbol{\alpha}}(s)\Big|
ds\Big),
& (\gamma_1,\gamma_2,\gamma_3)\in \mathbb{R}^3_{---},\;\, \beta\in ]\gamma_1,0[,
\\
\displaystyle
O\Big(\int_{t_0}^{\infty}e^{-\beta(t-s)}
\Big|\sum_{|\boldsymbol{\alpha}|=0}\Omega_{\boldsymbol{\alpha}}(s)\Big|
ds\Big),
& (\gamma_1,\gamma_2,\gamma_3)\in \mathbb{R}^3_{+--},\;\, \beta\in ]\gamma_2,0[,
\\
\displaystyle
O\Big(\int_{t_0}^{\infty}e^{-\beta(t-s)}
\Big|\sum_{|\boldsymbol{\alpha}|=0}\Omega_{\boldsymbol{\alpha}}(s)\Big|
ds\Big),
& (\gamma_1,\gamma_2,\gamma_3)\in \mathbb{R}^3_{++-},\;\, \beta\in ]\gamma_3,0[,
\\
\displaystyle
O\Big(\int_{t_0}^{t}e^{-\beta(t-s)}
\Big|\sum_{|\boldsymbol{\alpha}|=0}\Omega_{\boldsymbol{\alpha}}(s)\Big|
ds\Big),
& (\gamma_1,\gamma_2,\gamma_3)\in \mathbb{R}^3_{+++},\;\, \beta\in ]0,\gamma_3[.
\end{array}
\right.
\;\;\;{}
\label{eq:asymptotic_teo:form}
\end{eqnarray}
\end{theorem}

\begin{proof}
We prove the formula \eqref{eq:asymptotic_teo:form} by analyzing 
an iterative sequence and
using the properties of the operator $T$ defined in \eqref{eq:operator_fix_point}.

\vspace{0.5cm}
\noindent
{\it Proof of \eqref{eq:asymptotic_teo:form}
for  $(\gamma_1,\gamma_2,\gamma_3)\in \mathbb{R}^3_{---}$.}
Let us denote by  $T$ the operator defined in \eqref{eq:operator_fix_point}.
Now,  on $D_\upeta$ with $\upeta\in ]0,1[,$ we 
define the sequence $\omega_{n+1}=T\omega_{n}$ with $\omega_0=0$, we have
that $\omega_n\to z$ when $n\to\infty$. This fact is a consequence 
of the contraction property of $T$. 

We note that the Green
function $g$ defined on \eqref{eq:green_function} is given in terms
of $g_1$, since 
$(\gamma_1,\gamma_2,\gamma_3)\in \mathbb{R}^3_{---}$.
Then, we have that  the operator $T$ can be rewritten equivalently
as follows
\begin{eqnarray}
Tz(t)=\frac{1}{(\gamma_2-\gamma_1)(\gamma_3-\gamma_2)(\gamma_3-\gamma_1)}
\int_{t}^{\infty} g_1(t,s)
\mathbb{P}\Big(s,z(s),z'(s),z''(s)\Big)ds,
\quad\mbox{for $t\ge t_0$,}
\label{eq:operator_fix_point_g1}
\end{eqnarray}
since $g_1(t,s)=0$ for $s\in [t_0,t]$.
Thus,
the proof of \eqref{eq:asymptotic_teo:form} 
is reduced to prove that
\begin{align}
&\exists \; \Phi_n\in \mathbb{R}_+ \; :\; 
|\omega_n(t)|+|\omega'_n(t)|+|\omega''_n(t)|
\leq
\Phi_n\int_{t}^{\infty}e^{-\beta(t-\tau)}
\Big|\sum_{|\boldsymbol{\alpha}|=0}\Omega_{\boldsymbol{\alpha}}(\tau)\Big|
d\tau,
\;\text{ $\forall$ }t\geq t_0,
\quad\mbox{${}$}
\label{eq17}
\\
& \exists\; \Phi\in\mathbb{R}_+\;\;\; :\;  \Phi_n\to \Phi,
\mbox{ when $n\to \infty$}.
\label{eq17:uniformly_bounded}
\end{align}
Hence,
to complete the proof of \eqref{eq:asymptotic_teo:form} with $i=1$, 
we proceed to prove \eqref{eq17} by mathematical induction on $n$
and deduce that \eqref{eq17:uniformly_bounded} is a consequence of the construction
of the sequence $\{\Phi_n\}$.

We now prove \eqref{eq17}. Note that
for $n=1$ the estimate \eqref{eq17} is satisfied 
with $\Phi_1=\hat{A}$. Indeed,
it can be proved immediately by the definition of the operator $T$
given on \eqref{eq:operator_fix_point_g1}, the property  
$\mathbb{P}(s,0,0,0)=0$,
the estimate \eqref{eq:bound:g1} and the hypothesis
that $\beta\in [\gamma_1,0[$, since
\begin{eqnarray*}
|\omega_{1}(t)|+|\omega'_{1}(t)|+ |\omega''_{1}(t)|
  &=& 
|T\omega_0(t)|+|T'\omega_0(t)|+|T''\omega_0(t)|
  \\
&=&\frac{1}{|(\gamma_2-\gamma_1)(\gamma_3-\gamma_2)(\gamma_3-\gamma_1)|}
\\
&&\hspace{0.5cm}
	\times
	\int_{t}^{\infty}
	\left(
	\left| g_1(t,s)\right|+
	\left|\frac{\partial g_1}{\partial t}(t,s)\right|+
	\left|\frac{\partial^2 g_1}{\partial t^2}(t,s)\right|
	\right)
	\Big|\sum_{|\boldsymbol{\alpha}|=0}\Omega_{\boldsymbol{\alpha}}(s)\Big|ds
 \\
&\le &
	\frac{A}{|(\gamma_2-\gamma_1)(\gamma_3-\gamma_2)(\gamma_3-\gamma_1)|}
	\int_{t}^{\infty}e^{-\gamma_1(t-s)}
	\Big|\sum_{|\boldsymbol{\alpha}|=0}\Omega_{\boldsymbol{\alpha}}(s)\Big|
	ds
 \\
&\le &	
	\hat{A}
	\int_{t}^{\infty}e^{-\beta(t-\tau)}
	\Big|\sum_{|\boldsymbol{\alpha}|=0}\Omega_{\boldsymbol{\alpha}}(\tau)\Big|
	d\tau.
\end{eqnarray*}
Now, assuming that \eqref{eq17} is valid for  
$n=k$, we prove that \eqref{eq17} is also valid for $n=k+1$.
However, before to prove the estimate \eqref{eq17} for $n=k+1$, we note that 
by \eqref{eq:general_no_lin:2} and the fact
that $\max\{1,\upeta,\upeta^2,\upeta^3\}=1$, we deduce the following estimate
\begin{eqnarray}
&&\Big|\mathbb{P}\Big(s,\omega_k(s),\omega'_k(s),\omega''_k(s)\Big)\Big|
\le
\sum_{|\boldsymbol{\alpha}|=0}^{4}
\Big|\Omega_{\boldsymbol{\alpha}}(s)\Big|
\Big|\omega_k(s)\Big|^{\alpha_1}
\Big|\omega'_k(s)\Big|^{\alpha_2}
\Big|\omega''_k(s)\Big|^{\alpha_3}
\nonumber\\
&&
\hspace{1.5cm}
\le 
\sum_{|\boldsymbol{\alpha}|=0}
\Big|\Omega_{\boldsymbol{\alpha}}(s)\Big|
+
\Big(|\omega_k(s)|+|\omega'_k(s)|+|\omega''_k(s)|\Big)
\left(\sum_{|\boldsymbol{\alpha}|=1}^{4}
\Big|\Omega_{\boldsymbol{\alpha}}(s)\Big|
\upeta^{|\boldsymbol{\alpha}|-1}
\right)
\nonumber
\\
&&
\hspace{1.5cm}
\le 
\sum_{|\boldsymbol{\alpha}|=0}
\Big|\Omega_{\boldsymbol{\alpha}}(s)\Big|
+
\Big(|\omega_k(s)|+|\omega'_k(s)|+|\omega''_k(s)|\Big)
\sum_{|\boldsymbol{\alpha}|=1}^{4}
\Big|\Omega_{\boldsymbol{\alpha}}(s)\Big|.
\label{eq:efe_estimate}
\end{eqnarray}
Using \eqref{eq:operator_fix_point_g1},
 the inductive hypothesis,
the inequality \eqref{eq:bound:g1} and the estimate \eqref{eq:efe_estimate}
we have that
\begin{eqnarray*}
&&|\omega_{k+1}(t)|+|\omega'_{k+1}(t)|+|\omega''_{k+1}(t)|
\\
&& 
\hspace{0.5cm}
= |T\omega_k(t)|+|T'\omega_k(t)|+|T''\omega_k(t)|
 \\
 && 
\hspace{0.5cm} 
 =
 \frac{1}{|(\gamma_2-\gamma_1)(\gamma_3-\gamma_2)(\gamma_3-\gamma_1)|}\Bigg\{\left
 |\int_{t}^{\infty}
	g_1(t,s)
	\mathbb{P}\Big(s,\omega_k(s),\omega'_k(s),\omega''_k(s)\Big)
	ds\right|
\\
&& 
\hspace{1.7cm}
	 +
 \left|\int_{t}^{\infty}
	\frac{\partial g_1}{\partial t}(t,s)
	\mathbb{P}\Big(s,\omega_k(s),\omega'_k(s),\omega''_k(s)\Big)
	ds\right|
\\
&& 
\hspace{1.7cm}
 +\left|\int_{t}^{\infty}
	\frac{\partial^2 g_1}{\partial t^2}(t,s)
	\mathbb{P}\Big(s,\omega_k(s),\omega'_k(s),\omega''_k(s)\Big)
	ds\right|
	 \Bigg\}
\\
&&\hspace{0.5cm} 
 =\frac{1}{|(\gamma_2-\gamma_1)(\gamma_3-\gamma_2)(\gamma_3-\gamma_1)|}
\\
&&\hspace{1cm} 
\times
 \int_{t}^{\infty}
	\left(
	\left| g_1(t,s)\right|
	+\left|\frac{\partial g_1}{\partial t}(t,s) \right|
	+\left|\frac{\partial^2 g_1}{\partial t^2}(t,s) \right|
	\right)
	  \Big|
	\mathbb{P}\Big(s,\omega_k(s),\omega'_k(s),\omega''_k(s)\Big)
	\Big|ds
\\
&&\hspace{0.5cm}
\le
\hat{A}\int_{t}^{\infty}e^{-\gamma_1(t-s)}
\left\{
\sum_{|\boldsymbol{\alpha}|=0}
\Big|\Omega_{\boldsymbol{\alpha}}(s)\Big|
+
\Big(|\omega_k(s)|+|\omega'_k(s)|+|\omega''_k(s)|\Big)
\sum_{|\boldsymbol{\alpha}|=1}^{4}
\Big|\Omega_{\boldsymbol{\alpha}}(s)\Big|
\right\}
ds
\\
&&\hspace{0.5cm}
\le
\hat{A}\int_{t}^{\infty}e^{-\gamma_1(t-s)}
\left\{
\sum_{|\boldsymbol{\alpha}|=0}
\Big|\Omega_{\boldsymbol{\alpha}}(s)\Big|
+
\sum_{|\boldsymbol{\alpha}|=1}^{4}
\Big|\Omega_{\boldsymbol{\alpha}}(s)\Big|
\Phi_k\int_{s}^{\infty}e^{-\gamma_1(s-\tau)}
\sum_{|\boldsymbol{\alpha}|=0}
\Big|\Omega_{\boldsymbol{\alpha}}(\tau)\Big|
d\tau 
\right\}
ds
\\
&&\hspace{0.5cm}
\le
\hat{A}
\left\{ 1+
\int_{t}^{\infty}e^{-\gamma_1(t-s)}
\sum_{|\boldsymbol{\alpha}|=1}^{4}
\Big|\Omega_{\boldsymbol{\alpha}}(s)\Big|
\Phi_kds
\right\}
\int_{t}^{\infty}e^{-\gamma_1(t-\tau)}
\sum_{|\boldsymbol{\alpha}|=0}
\Big|\Omega_{\boldsymbol{\alpha}}(\tau)\Big|
d\tau 
\\
&&\hspace{0.5cm}
=
\hat{A}
\left( 
1+\frac{\Phi_k\rho}{-\gamma_1+\beta}
\right)
\int_{t}^{\infty}e^{-\beta(t-\tau)}
\sum_{|\boldsymbol{\alpha}|=0}
\Big|\Omega_{\boldsymbol{\alpha}}(\tau)\Big|
d\tau.
\end{eqnarray*}
Then, by the induction process, \eqref{eq17} is satisfied with 
$\Phi_{n}=   \hat{A}(1+\Phi_{n-1}\;\rho\;(-\gamma_1+\beta)^{-1}).$
By \eqref{eq:asymptotic_teo:sigma} we can rewrite $\Phi_{n}$
as follows $\Phi_{n}=   \hat{A}(1+\Phi_{n-1}\;\rho\;\upvarsigma)$

The proof of \eqref{eq17:uniformly_bounded} is given as follows. 
Using recursively the definition of $\Phi_{n-2},\ldots,\Phi_{2}$,
we can rewrite $\Phi_{n}$ as the sum of the terms of 
a geometric progression where the common ratio is given
by $\rho \hat{A}\upvarsigma$. Then, 
the hypothesis \eqref{eq:hipotesis3} implies the existence
of $\Phi$ satisfying \eqref{eq17:uniformly_bounded},
since by the construction of 
$\rho$ we have that
$\rho\hat{A}\upvarsigma\in ]0,1[.$
More precisely, we deduce that         
\begin{eqnarray*}
\lim_{n\to\infty}
\Phi_{n}=\hat{A}\lim_{n\to\infty}\sum_{i=0}^{n-1}
	\Big(\rho\hat{A}\upvarsigma\Big)^i
        =\hat{A}
        \lim_{n\to\infty}
        \frac{\Big[(\rho\hat{A}\upvarsigma)^n-1\Big]}
        {\rho\hat{A}|\gamma_1|^{-1}-1}
        =
        \frac{\hat{A}}{1-\rho\hat{A}\upvarsigma}
        =\Phi>0.
\end{eqnarray*}

Hence, \eqref{eq17}-\eqref{eq17:uniformly_bounded} are  valid and
the proof of  \eqref{eq:asymptotic_teo:form}
for  $(\gamma_1,\gamma_2,\gamma_3)\in \mathbb{R}^3_{---}$
is concluded by passing to the limit the sequence $\{\Phi_{n}\}$ 
when $n\to\infty$ 
in the topology of $C^2_0([t_0,\infty]).$

\vspace{0.5cm}
\noindent
{\it Proof of \eqref{eq:asymptotic_teo:form}
for  $(\gamma_1,\gamma_2,\gamma_3)\in \mathbb{R}^3_{+--}$.}
Similarly to the case $(\gamma_1,\gamma_2,\gamma_3)\in \mathbb{R}^3_{---}$ we 
define the sequence $\omega_{n+1}=T\omega_{n}$ with $\omega_0=0$ and,
by the contraction property of $T$, we can deduce that
$\omega_n\to z$ when $n\to\infty$. 
Then, the Green function $g$ defined on \eqref{eq:green_function} is given in
terms of $g_2$.
Thereby,  the operator $T$ can be rewritten equivalently
as follows
\begin{eqnarray}
Tz(t)&=&\frac{1}{(\gamma_2-\gamma_1)(\gamma_3-\gamma_2)(\gamma_3-\gamma_1)}
\int_{t_0}^{\infty} g_2(t,s)
\mathbb{P}\Big(s,\omega_k(s),\omega'_k(s),\omega''_k(s)\Big)ds
\nonumber\\
&=&
\frac{1}{(\gamma_2-\gamma_1)(\gamma_3-\gamma_2)(\gamma_3-\gamma_1)} 
\Bigg\{
\int_{t_0}^{t} 
	(\gamma_2-\gamma_3)e^{-\gamma_1(t-s)}
	\mathbb{P}\Big(s,\omega_k(s),\omega'_k(s),\omega''_k(s)\Big)ds
\nonumber\\
&&
+\int_{t}^{\infty} 
	\Big[(\gamma_1-\gamma_2)e^{-\gamma_3(t-s)}
	+(\gamma_3-\gamma_1)e^{-\gamma_2(t-s)}\Big]
\mathbb{P}\Big(s,\omega_k(s),\omega'_k(s),\omega''_k(s)\Big)ds
\Bigg\}.
\nonumber\\
&&
\label{eq:operator_fix_point_g2}
\end{eqnarray}
Then, the proof of \eqref{eq:asymptotic_teo:form} 
for  $(\gamma_1,\gamma_2,\gamma_3)\in \mathbb{R}^3_{+--}$
is reduced to prove 
\begin{align}
&\exists\quad \Phi_n\in \mathbb{R}_+ \; :\; 
|\omega_n(t)|+|\omega'_n(t)|+|\omega''_n(t)|\leq
\Phi_n\int_{t_0}^{\infty}e^{-\beta(t-\tau)}
\Big|\sum_{|\boldsymbol{\alpha}|=0}\Omega_{\boldsymbol{\alpha}}(\tau)\Big|
d\tau,\;\text{ $\forall$ }t\geq t_0,
\label{eq17_2}
\\
&\exists\quad\Phi\in\mathbb{R}_+\;\;\; :\; \Phi_n\to \Phi
\mbox{ when $n\to \infty$}.
\label{eq17:uniformly_bounded_2}
\end{align}
In the induction step for $n=1$ the estimate \eqref{eq17_2} is satisfied 
with $\Phi_1=\hat{A}$, 
since by the definition of the operator $T$
given on \eqref{eq:operator_fix_point_g2}, the property  
$\mathbb{P}(s,0,0,0)=0$,
the estimate \eqref{eq:bound:gi}
and the fact
that $\beta\in[\gamma_2,0[\subset[\gamma_3,\gamma_1]$, 
we deduce the following bound
\begin{eqnarray*}
 && |\omega_{1}(t)|+ |\omega'_{1}(t)|+|\omega''_{1}(t)| 
  =|T\omega_0(t)|+|T'\omega_0(t)|+|T''\omega_0(t)|
\\
&&\hspace{1.5cm}
\le
\frac{1}{|(\gamma_2-\gamma_1)(\gamma_3-\gamma_2)(\gamma_3-\gamma_1)|}
\\
&&\hspace{2cm}
\times
\Bigg\{
|\gamma_2-\gamma_3|\Big(1+|\gamma_1|+|\gamma_1|^2\Big)
\int_{t_0}^{t}e^{-\gamma_1(t-s)}
\Big|\sum_{|\boldsymbol{\alpha}|=0}\Omega_{\boldsymbol{\alpha}}(\tau)\Big|
ds
\\
&&\hspace{2cm}
\Big[
|\gamma_2-\gamma_1|\Big(1+|\gamma_3|+|\gamma_3|^2\Big)
+|\gamma_3-\gamma_1|\Big(1+|\gamma_2|+|\gamma_2|^2\Big)
\Big]
\\
&&\hspace{2cm}
\times\int_{t}^{\infty}e^{-\beta(t-s)}
\Big|\sum_{|\boldsymbol{\alpha}|=0}\Omega_{\boldsymbol{\alpha}}(\tau)\Big|
ds\Bigg\}
\\
&&\hspace{1.5cm}
\le
\hat{A}
\left\{
\int_{t_0}^{t}e^{-\max\{\gamma_2,\gamma_3\}(t-s)}
\Big|\sum_{|\boldsymbol{\alpha}|=0}\Omega_{\boldsymbol{\alpha}}(s)\Big|
ds
+
\int_{t}^{\infty}e^{-\beta(t-s)}
\Big|\sum_{|\boldsymbol{\alpha}|=0}\Omega_{\boldsymbol{\alpha}}(s)\Big|
ds
\right\}
\\
&&\hspace{1.5cm}
=\hat{A}\int_{t_0}^{\infty}e^{-\beta(t-\tau)}
\Big|\sum_{|\boldsymbol{\alpha}|=0}\Omega_{\boldsymbol{\alpha}}(\tau)\Big|
d\tau.
\end{eqnarray*}
Then, the general induction step can be proved as follows
\begin{eqnarray*}
&&  |\omega_{k+1}(t)|+|\omega'_{k+1}(t)|+|\omega''_{k+1}(t)| 
\\
&&\qquad
  = |T\omega_k(t)|+|T'\omega_k(t)|+ |T''\omega_k(t)|
	\\
&&\qquad\le 
	\hat{A}
	\int_{t_0}^{t}e^{-\gamma_1(t-s)}
	\left\{
	\Big|\sum_{|\boldsymbol{\alpha}|=0}\Omega_{\boldsymbol{\alpha}}(s)\Big|
	+
	\Big(|\omega_k(s)|+|\omega'_k(s)|+|\omega''_k(s)|\Big)
	\sum_{|\boldsymbol{\alpha}|=1}^4
	\Big|\Omega_{\boldsymbol{\alpha}}(s)\Big|
	 \right\}
	ds
	\\
&&\qquad
  	+\hat{A}
	\int_{t}^{\infty}e^{-\gamma_2(t-s)}
	\left\{
	\Big|\sum_{|\boldsymbol{\alpha}|=0}\Omega_{\boldsymbol{\alpha}}(s)\Big|
	+
	\Big(|\omega_k(s)|+|\omega'_k(s)|+|\omega''_k(s)|\Big)
	\sum_{|\boldsymbol{\alpha}|=1}^4
	\Big|\Omega_{\boldsymbol{\alpha}}(s)\Big|
	\right\}
	ds
	\\
\\&&\qquad=
	\hat{A} \Big[J_1(t)+J_2(t)\Big]
\\
&&\qquad
\le 
   \hat{A}\Big(
	1
	+\Phi_k\;\rho\;
	\upvarsigma\Big)
	\int_{t_0}^{\infty} e^{-\beta (t-\tau)}
	\Big|\sum_{|\boldsymbol{\alpha}|=0}\Omega_{\boldsymbol{\alpha}}(\tau)\Big|
	d\tau,
\end{eqnarray*} 
where $\upvarsigma$ is the number defined on \eqref{eq:asymptotic_teo:sigma},
since for $\beta\in]\gamma_2,0[\subset]\gamma_3,\gamma_1[$,  we can deduce that 
\begin{eqnarray*}
	J_1(t)&:=&\int_{t_0}^{t}e^{-\gamma_1(t-s)}
	\Big|\sum_{|\boldsymbol{\alpha}|=0}\Omega_{\boldsymbol{\alpha}}(s)\Big|
	ds
	+
	\int_{t}^{\infty}e^{-\gamma_2(t-s)}
	\Big|\sum_{|\boldsymbol{\alpha}|=0}\Omega_{\boldsymbol{\alpha}}(s)\Big|
	ds
	\\
	&\le&
	\int_{t_0}^{t}e^{-\beta(t-s)}
	\Big|\sum_{|\boldsymbol{\alpha}|=0}\Omega_{\boldsymbol{\alpha}}(s)\Big|
	ds
	+
	\int_{t}^{\infty}e^{-\beta(t-s)}
	\Big|\sum_{|\boldsymbol{\alpha}|=0}\Omega_{\boldsymbol{\alpha}}(s)\Big|
	ds
	\\
	&=&
	\int_{t_0}^{\infty}e^{-\beta(t-\tau)}
	\Big|\sum_{|\boldsymbol{\alpha}|=0}\Omega_{\boldsymbol{\alpha}}(\tau)\Big|
	d\tau,
\\
	J_2(t)&:=&
	\int_{t_0}^t
	e^{-\gamma_1(t-s)}
	\Big(|\omega_k(s)|+|\omega'_k(s)|+|\omega''_k(s)|\Big)
	\sum_{|\boldsymbol{\alpha}|=1}^4\Big|\Omega_{\boldsymbol{\alpha}}(s)\Big|
	ds
	\\
	&&
	+	\int_{t}^\infty
	e^{-\gamma_2(t-s)}
	\Big(|\omega_k(s)|+|\omega'_k(s)|+|\omega''_k(s)|\Big)
	\sum_{|\boldsymbol{\alpha}|=1}^4\Big|\Omega_{\boldsymbol{\alpha}}(s)\Big|
	ds
\\
&\le&
	\rho\int_{t_0}^t
	e^{-\gamma_1(t-s)}
	\Big(|\omega_k(s)|+|\omega'_k(s)|+|\omega''_k(s)|\Big)
	ds
\\
&&
	+\rho	\int_{t}^\infty
	e^{-\gamma_2(t-s)}
	\Big(|\omega_k(s)|+|\omega'_k(s)|+|\omega''_k(s)|\Big)
	ds
\\
	&\le&
	\Phi_k\;\rho \left(
	\frac{1-\exp(-\sigma_1(t_0-t))}{-(-\sigma_1+\beta)}
	+\frac{1}{-\sigma_2+\beta}
	\right)\;
	\int_{t_0}^{\infty}e^{-\beta(t-\tau)}
	\Big|\sum_{|\boldsymbol{\alpha}|=0}\Omega_{\boldsymbol{\alpha}}(\tau)\Big|
	d\tau
\\
	&\le&
	\Phi_k\;\rho \;\upvarsigma
	\int_{t_0}^{\infty}e^{-\beta(t-\tau)}
	\Big|\sum_{|\boldsymbol{\alpha}|=0}\Omega_{\boldsymbol{\alpha}}(\tau)\Big|
	d\tau.
\end{eqnarray*}
Hence the thesis of the inductive steps holds with 
$\Phi_{n}=\hat{A}(1+\Phi_{n-1}\rho\;\upvarsigma)$.

We proceed in an analogous way to the case  
$(\gamma_1,\gamma_2,\gamma_3)\in \mathbb{R}^3_{---}$ 
and deduce that \eqref{eq17:uniformly_bounded_2} is satisfied
with $\Phi=\hat{A}/(1-\rho\;\upvarsigma\hat{A})>0$.

Therefore, the sequence $\{\Phi_{n}\}$ is convergent 
and $z_2$ (the limit of $\omega_n$
in the topology of $C^2_0([t_0,\infty])$) satisfies \eqref{eq:asymptotic_teo:form}.

\vspace{0.5cm}
\noindent
{\it Proof of \eqref{eq:asymptotic_teo:form}
for  $(\gamma_1,\gamma_2,\gamma_3)\in \mathbb{R}^3_{++-}$ 
and for  $(\gamma_1,\gamma_2,\gamma_3)\in \mathbb{R}^3_{+++}$.}
The proof of \eqref{eq:asymptotic_teo:form}
for  $(\gamma_1,\gamma_2,\gamma_3)\in \mathbb{R}^3_{++-}$ 
and $(\gamma_1,\gamma_2,\gamma_3)\in \mathbb{R}^3_{+++}$
are completely analogous
to the proofs of \eqref{eq:asymptotic_teo:form}
for $(\gamma_1,\gamma_2,\gamma_3)\in \mathbb{R}^3_{---}$ 
and $(\gamma_1,\gamma_2,\gamma_3)\in \mathbb{R}^3_{+--}$, respectively.
\end{proof}

\section{$L^p$-solutions for \eqref{eq:general_no_lin}}
\label{section:lpsolution}

\begin{theorem}
\label{teo:solution_integrab}
Let us consider that the hypotheses of Theorem~\ref{teo:solution_asymptotic}
are satisfied
and denote by $\lceil \cdot \rceil$ the ceiling function and
by $W^{2,p}([t_0,\infty[)$ the Sobolev space  defined by
 \begin{eqnarray}
 \lceil x \rceil&=& n+1, \quad x\in]n,n+1],\quad n\in\mathbb{Z},\quad\mbox{and}
 \label{eq:ceiling_function}
 \\
W^{2,p}([t_0,\infty[)&=&\Big\{u\in L^{p}([t_0,\infty[)
\quad:\quad 
u',u''\in L^{p}([t_0,\infty[)\Big\},
\label{eq:sobolev_space}
\end{eqnarray} 
respectively. 
Moreover assume that the condition
\begin{enumerate}
 \item[($P_3$)] 
The coefficients $\Omega_{\boldsymbol{\alpha}}$ of $\mathbb{P}$
are such that $\Omega_{\boldsymbol{\alpha}}\in L^p([t_0,\infty[)$
for $|\boldsymbol{\alpha}|=0$ and  for  $|\boldsymbol{\alpha}|\ge 1$
the functions  $\Omega_{\boldsymbol{\alpha}}$ are of the following type
$
 \Omega_{\boldsymbol{\alpha}}(t)=
 \lambda_{\boldsymbol{\alpha},p}\Omega_{\boldsymbol{\alpha},p}(t)
 +\lambda_{\boldsymbol{\alpha},c},
$
where $\lambda_{\boldsymbol{\alpha},p}$ and $\lambda_{\boldsymbol{\alpha},c}$ are real constants
and $\Omega_{\boldsymbol{\alpha},p}\in L^p([t_0,\infty[)$.
\end{enumerate}
is satisfied.
Then,  the following assertions are valid for $z$, 
the  solution of~\eqref{eq:general_no_lin}, 
\begin{enumerate}
 \item[(i)] $z$  is a function belongs to the Sobolev space $W^{2,p}([t_0,\infty[).$
  \item[(ii)] Let $m$ the number defined by
$m=\lceil p \rceil -1$ for $p\in ]1,4]$
  and by $m=4$ for $p\in ]4,\infty]$.  There exists $m+1$ functions, denoted by
  $\Theta_{1},\ldots,\Theta_{m}$ and $\Psi,$  
  such that $\Theta_{k}\in W^{2,p/k}([t_0,\infty[),$
  for $k=1,\ldots,m$, 
  $\Psi\in W^{2,1}([0,\infty[)$ and
  $z^{(\ell)}(t)=\sum_{k=1}^m\Theta^{(\ell)}_{k}(t)+\Psi^{(\ell)}(t)$
  for $\ell=0,1,2.$
\end{enumerate}
\end{theorem}

\begin{proof}
\noindent
{\it (i).}
Let us denote by $\gamma_\beta$ and $\mathbb{Y}$ the functions defined as follows
\begin{eqnarray*}
\gamma_{\beta}(t)&=&
\left\{
\begin{array}{ll}
 \exp(-\beta t) & t\ge t_0,\\
 0&\mbox{elsewhere},
\end{array}
\right.
\qquad
\mbox{ for some $t_0\ge 0$,}
\\
\mathbb{Y}(t)&=&\int_{t_0}^\infty \exp(-\beta(t-s))
\sum_{|\boldsymbol{\alpha}|=0}\Omega_{\boldsymbol{\alpha}}(s)
ds.
\end{eqnarray*}
We note that 
$\mathbb{Y}
=\gamma_{\beta}\ast \sum_{|\boldsymbol{\alpha}|=0}\Omega_{\boldsymbol{\alpha}}$,
where $\ast$ denotes the convolution.
Then, by the convolution properties and the hypothesis
that the coefficients $\Omega_{\boldsymbol{\alpha}}$
of $\mathbb{P}$ with $|\boldsymbol{\alpha}|=0$ 
are belong of $L^p([t_0,\infty[)$,  we follow
that $\mathbb{Y}$ is belongs to $L^{p}([t_0,\infty[)$.
Thus, by \eqref{eq:asymptotic_teo:form},  we deduce that
$z,z',z'' \in L^{p}([t_0,\infty[)$ 
or equivalently by \eqref{eq:sobolev_space}
$z$ is belongs ${W^{2,p}}([0,\infty[).$

\vspace{0.5cm} 
\noindent
{\it (ii).} 
By  \eqref{eq:operator_fix_point}-\eqref{eq:operator_equation}
and \eqref{eq:general_no_lin:2},
we have that $z_i$ can be rewritten as follows
\begin{eqnarray}
z(s)&=&\int_{t_0}^\infty g(t,s)\mathbb{P}\Big(s,z(s),z'(s),z''(s)\Big)ds
\nonumber\\
&=&\sum_{k=0}^4\int_{t_0}^\infty g(t,s)
\sum_{|\boldsymbol{\alpha}|=k}\Omega_{\boldsymbol{\alpha}}(s)
[z(s)]^{\alpha_1}[z'(s)]^{\alpha_2}[z''(s)]^{\alpha_3}ds
:=\sum_{k=0}^4\mathcal{I}_k(s).
\label{eq:descomp_teo:solution_integrab}
\end{eqnarray}
Now, we apply the hypothesis ($P_3$) to construct the functions
$\Theta^{i}$ and $\Psi$. 

If $|\boldsymbol{\alpha}|=0,$
by ($P_3$) and similar arguments to those used in the proof
of item (i) we have that $\mathcal{I}_0\in W^{2,p}([t_0,\infty[)$.
Moreover, if $|\boldsymbol{\alpha}|=k\in\{1,2,3,4\}$,
by application of ($P_3$) and item~{\it (i)} we deduce that
$z,z',z''\in W^{2,p}([t_0,\infty[)$, which implies that
$[z]^{\alpha_1}[z']^{\alpha_2}[z'']^{\alpha_3}
\in L^{p/k}([t_0,\infty[)$. Now, by ($P_3$) for $|\boldsymbol{\alpha}|>1$
we have that
\begin{eqnarray}
\mathcal{I}_k(s)
&=&\int_{t_0}^\infty g(t,s)
\sum_{|\boldsymbol{\alpha}|=k}\Big(
 \lambda_{\boldsymbol{\alpha},p}\Omega_{\boldsymbol{\alpha},p}(s)
 +\lambda_{\boldsymbol{\alpha},c}\Big)
[z(s)]^{\alpha_1}[z'(s)]^{\alpha_2}[z''(s)]^{\alpha_3}ds
\nonumber\\
&=&\sum_{|\boldsymbol{\alpha}|=k}
\lambda_{\boldsymbol{\alpha},p}
\int_{t_0}^\infty g(t,s)
 \Omega_{\boldsymbol{\alpha},p}(s)
[z(s)]^{\alpha_1}[z'(s)]^{\alpha_2}[z''(s)]^{\alpha_3}ds
\nonumber\\
&&
+
\sum_{|\boldsymbol{\alpha}|=k}
\lambda_{\boldsymbol{\alpha},c}
\int_{t_0}^\infty g(t,s)
[z(s)]^{\alpha_1}[z'(s)]^{\alpha_2}[z''(s)]^{\alpha_3}ds
\nonumber\\
&:=& \mathcal{I}_{k,p}(s)+\mathcal{I}_{k,c}(s),
\label{eq:jp_descomposition}
\end{eqnarray}
i.e. $\mathcal{I}_k$ can be rewritten as the linear combination of 
the functions $\mathcal{I}_{k,p}\in W^{2,p/(1+k)}([t_0,\infty[)$ and  
$\mathcal{I}_{k,c}\in W^{2,p/k}([t_0,\infty[)$.

By \eqref{eq:descomp_teo:solution_integrab} and \eqref{eq:jp_descomposition} 
we have that $z$ can be discomposed 
\begin{eqnarray*}
 z(s)&=&\mathcal{I}_{0}(s)+\sum_{k=1}^{4}\mathcal{I}_{k,p}(s)
 +\sum_{k=1}^{4}\mathcal{I}_{k,p}(s)
 \\
 &=&\Big(\Big[\mathcal{I}_{0}+\mathcal{I}_{1,c}\Big]
 +\Big[\mathcal{I}_{1,p}+\mathcal{I}_{2,c}\Big]
 +\Big[\mathcal{I}_{2,p}+\mathcal{I}_{3,c}\Big]
 +\Big[\mathcal{I}_{3,p}+\mathcal{I}_{4,c}\Big]
 +\mathcal{I}_{4,p}\Big)(s),
\\
 &=&\Big(\mathcal{H}_1
 +\mathcal{H}_2
 +\mathcal{H}_3
 +\mathcal{H}_4
 +\mathcal{I}_{4,p}\Big)(s),
\end{eqnarray*}
with 
\begin{eqnarray*}
&&\mathcal{H}_1:=\mathcal{I}_{0}+\mathcal{I}_{1,c}\in W^{2,p}([t_0,\infty[),\quad
\mathcal{H}_2:=\mathcal{I}_{1,p}+\mathcal{I}_{2,c}\in W^{2,p/2}([t_0,\infty[),\quad
\\
&&\mathcal{H}_3:=\mathcal{I}_{2,p}+\mathcal{I}_{3,c}\in W^{2,p/3}([t_0,\infty[),\quad
\mathcal{H}_4:=\mathcal{I}_{3,p}+\mathcal{I}_{4,c}\in W^{2,p/4}([t_0,\infty[),
\\
&&\mathcal{I}_{4,p}\in W^{2,p/5}([t_0,\infty[).
\end{eqnarray*}
Then, we have that the functions $\Theta_{i}$ and $\Psi$
satisfying the requirements  of the Theorem are defined as follows
\begin{eqnarray*}
\begin{array}{llll}
 p\in ]1,2],&m=1, &
 \Theta_{k}=\mathcal{H}_k\mbox{ for } k=1, 
 & 
 \Psi =\mathcal{H}_2+\mathcal{H}_3+\mathcal{H}_4+\mathcal{I}_{4,p}
 \\
 p\in ]2,3],&m=2, & 
 \Theta_{k}=\mathcal{H}_k\mbox{ for } k=1,2,
 &
 \Psi =\mathcal{H}_3+\mathcal{H}_4+\mathcal{I}_{4,p},
 \\
 p\in ]3,4],&m=3, &
 \Theta_{k}=\mathcal{H}_k\mbox{ for } k=1,2,3,
 &
 \Psi =\mathcal{H}_4+\mathcal{I}_{4,p},
 \\
 p\in ]4,\infty[,&m=4, & 
 \Theta_{k}=\mathcal{H}_k\mbox{ for } k=1,2,3,4,
 &
 \Psi =\mathcal{I}_{4,p},
\end{array}
\end{eqnarray*}
Thus, the result is valid for all $p>1$ and the Theorem is proved.  
\end{proof}

\section{Applications}
\label{sec:applications}

\subsection{The Poincar\'e problem and Poincar\'e type result.} $ $

The fourth order linear differential equation of Poincar\'e type
is given by
\begin{eqnarray}
y^{({\rm iv})}
+[a_3+r_3 (t)] y'''
+[a_2+r_2 (t)] y''
+[a_1+r_1 (t)] y'
+[a_0+r_0 (t)] y=0,
\label{eq:intro_uno}
\end{eqnarray}
where $a_i$ are constants and 
$r_i$ are real-valued functions.
Note that \eqref{eq:intro_uno} is a perturbation of the following 
constant coefficient equation:
\begin{eqnarray}
y^{({\rm iv})}
+a_3 y'''
+a_2 y''
+a_1 y'
+a_0 y=0.
 \label{eq:intro_dos}
\end{eqnarray}
Now, let us consider the new variable $z$ of the following type 
\begin{eqnarray}
z(t)=\frac{y'(t)}{y(t)}-\mu 
\quad\mbox{or equivalently}\quad
y(t)=\exp\Big(\int_{t_0}^t (z(s)+\mu )ds\Big),
\label{eq:general_change_var}
\end{eqnarray}
where $y$ is a solution of \eqref{eq:intro_uno} and
$\mu$ is an arbitrary root of the characteristic polynomial associated 
to \eqref{eq:intro_dos}. 
Then, differentiating  $y$ in \eqref{eq:general_change_var}  and
replacing the results 
in \eqref{eq:intro_uno}, we deduce that $z$ is a solution of 
the following third order nonlinear equation
\begin{eqnarray}
&&z'''+[4\mu  +a_3]z''
    +[6\mu ^2+3a_3\mu  +a_2]z'
    +[4\mu  ^3+3\mu  ^2a_3+2\mu  a_2
    +a_1]z 
    \nonumber\\
    & & \qquad
    =-\big[ \mu^3r_3(t)+\mu  ^2r_2(t)+\mu  r(t)+r_0(t)\big]
    -\big[3\mu  ^2r_3(t)+2\mu  r_2(t)+r_1(t)\big]z
    \nonumber\\
    & & \qquad \quad   
    -\big[3\mu  r_3(t)+r_2(t)\big]z'-r_3(t)z''
    -\big[12\mu    +3a_3+3r_3(t)\big]zz'
    -4zz''
    \nonumber\\
    & & \qquad\quad
    -\big[6\mu  ^2+3\mu  a_3+a_2+3\mu  r_3(t)+r_2(t)]z^2
    -3(z')^2
    -6z^2z'
    -\big[4\mu  +r_3(t)\big]z^3
    -z^4.
    \hspace{1cm}\mbox{$ $}
   \label{eq:ricati_original}
\end{eqnarray}
Then, the analysis of original linear perturbed
equation of fourth order 
\eqref{eq:intro_uno} is translated to the analysis of a nonlinear
third order equation  \eqref{eq:ricati_original}.

We note that the equation \eqref{eq:ricati_original}
is of the type \eqref{eq:general_no_lin}, since the 
constant coefficients $b_i$ are 
\begin{subequations}
\label{eq:notation_ricc}
\begin{eqnarray}
\begin{array}{rclcrcl}
  b_0 =4\mu ^3+3\mu ^2a_3+2\mu  a_2+a_1, 
  \quad
  b_1 =6\mu ^2+3\mu  a_3+a_2,
  \quad
  b_2 = 4\mu +a_3,  
\end{array}
\label{eq:notation_ricc:1}
\end{eqnarray}
and the functions $\Omega_{\boldsymbol{\alpha}}$ 
defining the coefficients of the polynomial 
$\mathbb{R}$ are  given by 
\begin{eqnarray}
\Omega_{\boldsymbol{\alpha}}(t)
=\left\{
\begin{array}{lcl}
 -\big(\mu ^3r_3(t)+\mu ^2r_2(t)+\mu  r_1(t)+r_0(t)\big), &
 \quad & \boldsymbol{\alpha}=(0,0,0),
 \\
 -\big(3\mu  ^2r_3(t)+2\mu  r_2(t)+r_1(t)\big), & & \boldsymbol{\alpha}=(1,0,0),
 \\
 -\big(3\mu  r_3(t)+r_2(t)\big), & & \boldsymbol{\alpha}=(0,1,0),
 \\
 -r_3(t),& & \boldsymbol{\alpha}=(0,0,1),
 \\
 -\big(12\mu  +3a_3+r_3(t)\big),& & \boldsymbol{\alpha}=(1,1,0),
 \\
 -4,& & \boldsymbol{\alpha}=(1,0,1),
 \\
 -\big(6\mu  ^2+3\mu  a_3+a_2
  +r_2(t) +3\mu  r_3(t)\big),& & \boldsymbol{\alpha}=(2,0,0),
 \\
 -3,& & \boldsymbol{\alpha}=(0,2,0),
 \\
 -6,& & \boldsymbol{\alpha}=(2,1,0),  
 \\
-\big(4+r_3(t)\big),& & \boldsymbol{\alpha}=(3,0,0),
 \\
 -1,& & \boldsymbol{\alpha}=(4,0,0),
 \\
 0,& & \mbox{otherwise}.
\end{array}
\right.
\label{eq:notation_ricc:2}
\end{eqnarray}
\end{subequations}
Thus, we can apply the results of section~\ref{sect:analisis_of_general}.

\begin{theorem}
\label{lem:solution_perturbation}
Let us consider that the hypotheses 
\begin{enumerate}
 \item[(H$_1$)] The set of characteristic  roots for  \eqref{eq:intro_dos} 
 is $\Big\{\lambda_i, i=\overline{1,4}\;:\;
\lambda_1>\lambda_2>\lambda_3>\lambda_4\Big\}\subset\mathbb{R}$.

 \item[(H$_2$)] Let  $p:\mathbb{R}^2\to\mathbb{R}$ the function defined by
 \begin{eqnarray}
  p(\mu,s)=-\big(\mu ^3r_3(t)+\mu ^2r_2(t)+\mu  r_1(t)+r_0(t)\big).
  \label{eq:termino_constante}
 \end{eqnarray}
 The perturbation functions 
 are selected such that 
 $\mathcal{G}(p(\lambda_i,\cdot))(t)\to 0$ and 
 $\mathcal{L}(r_j)(t)\to 0$, $j=0,1,2,3$,
 when $t\to\infty$, 
where $\mathcal{G}$ and 
$\mathcal{L}$ are the functionals  defined on 
\eqref{eq:functyonal_G} and \eqref{eq:functyonal_L},
respectively.

\end{enumerate}
are satisfied.
Then, for each $i\in\{1,\ldots,4\}$, the equation~\eqref{eq:ricati_original} 
with $\mu=\lambda_i$ has a unique solution
$z_i$ such that $z_i\in C_0^2([t_0,\infty[)$.
\end{theorem}

\begin{proof} The prove is reduced to apply the Theorem~\ref{thm:general_no_lin}.
Indeed, in the following we prove that hypotheses   ($P_1$) and ($P_2$)
are satisfied.

The hypothesis (H$_1$) implies ($P_1$).
This result is a consequence of the following fact:
if  $\lambda_i$   and $\lambda_j$ are  two distinct  
characteristic roots of the polynomial associated to \eqref{eq:intro_dos},
then $\lambda_j-\lambda_i$ is a 
root of the characteristic polynomial associated with 
\eqref{eq:ricati_original} when the right hand side is zero.
Indeed, considering $\lambda_i\not = \lambda_j$ 
satisfying the characteristic polynomial associated to 
\eqref{eq:intro_dos}, subtracting the equalities, dividing the result
by $\lambda_j-\lambda_i$ and using the identities
\begin{eqnarray*}
  \lambda_j^3+\lambda_j^2\lambda_i+\lambda_j\lambda_i^2+\lambda_i^3 &=& (\lambda_j-\lambda_i)^3+4\lambda_i(\lambda_j-\lambda_i)^2+6\lambda_i^2(\lambda_j-\lambda_i)+4\lambda_i^3 \\
  a_3(\lambda_j^2+\lambda_j\lambda_i+\lambda_i^2) &=& a_3(\lambda_j-\lambda_i)^2+3a_3\lambda_i(\lambda_j-\lambda_i)+3a_3\lambda_i^2 \\
  a_2(\lambda_j-\lambda_i) &=&
  a_2(\lambda_j-\lambda_i)+2\lambda_ia_2,
\end{eqnarray*}
we deduce that $\lambda_j-\lambda_i$ is a 
root of the characteristic polynomial associated to 
\eqref{eq:ricati_original}.
Thus, if (H$_1$) holds we have that
\begin{subequations}
\label{eq:lambda_vs_gamma}
\begin{eqnarray}
\mbox{for } i=1 &:&0>\gamma_1=\lambda_2-\lambda_1>
	  \gamma_2=\lambda_3-\lambda_1>
	  \gamma_3=\lambda_4-\lambda_1
\label{eq:lambda_vs_gamma:1}\\
\mbox{for } i=2 &:&\gamma_1=\lambda_1-\lambda_2>0>
	  \gamma_2=\lambda_3-\lambda_2>
	  \gamma_3=\lambda_4-\lambda_2
\label{eq:lambda_vs_gamma:2}\\
\mbox{for } i=3 &:&\gamma_1=\lambda_1-\lambda_3>
	  \gamma_2=\lambda_2-\lambda_3>0>
	  \gamma_3=\lambda_4-\lambda_3
\label{eq:lambda_vs_gamma:3}\\
\mbox{for }i=4 &:&\gamma_1=\lambda_1-\lambda_4>
	  \gamma_2=\lambda_2-\lambda_4>
	  \gamma_3=\lambda_3-\lambda_4>0
\label{eq:lambda_vs_gamma:4}
\end{eqnarray} 
\end{subequations}
i.e. $\gamma_1>\gamma_2>\gamma_3$ and ($P_1$)
is valid. 
 
The hypothesis (H$_2$) implies ($P_2$). Indeed,
by \eqref{eq:notation_ricc:2} and \eqref{eq:termino_constante} we note that
\begin{eqnarray}
\mathcal{G}(\sum_{|\boldsymbol{\alpha}|=0}\Omega_{\boldsymbol{\alpha}})(t) 
&=&\mathcal{G}(-\big(\mu ^3r_3(\cdot)+\mu ^2r_2(\cdot)+\mu  r_1(\cdot)+r_0(\cdot)\big))(t)
\nonumber\\
&=&\mathcal{G}(p(\mu,\cdot))(t)
\label{eq:g_to_cero}\\
\mathcal{L}(\sum_{|\boldsymbol{\alpha}|=1}\Omega_{\boldsymbol{\alpha}})(t) 
&=&\mathcal{L}(-\big((3\mu^2+3\mu+1)r_3(\cdot)
	+(2\mu+1)r_2(\cdot)
	+r_1(\cdot)\big))(t)
\nonumber\\
&\le&|3\mu^2+3\mu+1|\mathcal{L}(r_3)(t)+|2\mu+1|\mathcal{L}(r_2)(t)
	+\mathcal{L}(r_1)(t)
\label{eq:L_to_cero}\\
\mathcal{L}(\sum_{|\boldsymbol{\alpha}|=2}^4\Omega_{\boldsymbol{\alpha}})(t) 
&=&\mathcal{L}(-\big((3\mu+2)r_3(\cdot)
	+r_2(\cdot)
	+r_1(\cdot)+3(1+\mu)a_3+a_2+6\mu^2+12\mu+18\big))(t)
\nonumber\\
&\le&|3\mu+2|\mathcal{L}(r_3)(t)+\mathcal{L}(r_2)(t)
\nonumber\\
&&
	+|3(1+\mu)a_3+a_2+6\mu^2+12\mu+18|\mathcal{L}(1)(t).
\label{eq:L_bounded}
\end{eqnarray}
Thus, clearly if (H$_2$) is satisfied, then  (P$_2$) is also satisfied.
\end{proof}

%
%

\begin{theorem}
\label{teo:poincare}
Let us assume that the hypothesis (H$_1$) and (H$_2$) are satisfied.
Denote by $W[y_1,\ldots,y_4]$ the Wronskian of $\{y_1,\ldots,y_4\}$.
Then, the equation \eqref{eq:intro_uno} has a fundamental system of 
solutions given by 
\begin{eqnarray}
    y_i(t)=\exp\Big(\int_{t_0}^{t}[\lambda_i+z_i(s)]ds\Big),
    \quad
    \mbox{with $z_i$ solution of \eqref{eq:ricati_original}
    with $\mu=\lambda_i$,}
    \quad i\in\{1,2,3,4\}.
    \label{eq:fundam_sist_sol}
\end{eqnarray}
 Moreover the following properties about
 the asymptotic behavior
\begin{eqnarray}
\frac{y'_i(t)}{y_i(t)}&=&\lambda_i,
\quad
\frac{y''_i(t)}{y_i(t)}=\lambda_i^2,
\quad
\frac{y'''_i(t)}{y_i(t)}=\lambda_i^3,
\quad
\frac{y_i^{({\rm iv})}(t)}{y_i(t)}=\lambda_i^4,
\label{eq:asymptotic_perturbed}
\\
W[y_1,\ldots,y_4]&=& 
\prod_{1\le k<\ell\le 4}\;\big(\lambda_{\ell}-\lambda_k\big)
\; y_1y_2y_3y_4\;\big(1+o(1)\big),
\label{eq:asymptotic_perturbed_w}
\end{eqnarray}
are satisfied when $t\to\infty.$
\end{theorem}

\begin{proof}
By Theorem~\ref{lem:solution_perturbation} and 
change of variable \eqref{eq:general_change_var},
we have that the fundamental system of solutions for
\eqref{eq:intro_uno} is given by \eqref{eq:fundam_sist_sol}.
Moreover, by \eqref{eq:fundam_sist_sol} we deduce the identities
\begin{eqnarray}
\frac{y'_i(t)}{y_i(t)}
	 &=&[\lambda_i+z_i(t)]
\label{eq:deri_uno}\\
\frac{y''_i(t)}{y_i(t)} 
	&=& [\lambda_i+z_i(t)]^2+z'_i(t), 
\label{eq:deri_dos}\\
\frac{y'''_i(t)}{y_i(t)} 
	&=&
[\lambda_i+z_i(t)]^3+3[\lambda_i+z_i(t)]z'_i(t)+z''_i(t),
\label{eq:deri_tres}\\
\frac{y_i^{({\rm iv})}(t)}{y_i(t)} 
	&=&
[\lambda_i+z_i(t)]^4
	+6[\lambda_i+z_i(t)]^2z'_i(t)
	+3[z'_i(t)]^2
\nonumber\\
	&&
	+4[\lambda_i+z_i(t)]z''_i(t)
	+z''_i(t),
\label{eq:deri_cuatro}
\end{eqnarray}
Now, using the fact that $z_i\in C_0^2([t_0,\infty[)$
is a solution of \eqref{eq:fundam_sist_sol} with
$\mu=\lambda_i$, we deduce
the proof of \eqref{eq:asymptotic_perturbed}.
Moreover by the definition of the $W[y_1,\ldots,y_4]$,
some algebraic rearrangements and 
\eqref{eq:asymptotic_perturbed}, we deduce \eqref{eq:asymptotic_perturbed_w}.
\end{proof}

\subsection{Levison type theorem.} $ $

Let us introduce some notations.
Consider the operators 
$\mathbb{F}_{1},\mathbb{F}_{2},\mathbb{F}_{3}$ and $\mathbb{F}_{4}$ 
defined as follows
\begin{eqnarray*}
\mathbb{F}_1(E)(t)
	&=&
	\int_{t}^{\infty}e^{-(\lambda_2-\lambda_1)(t-s)}|E(s)|ds,
\\
\mathbb{F}_2(E)(t)
	&=&
	\int_{t_0}^{t}e^{-(\lambda_1-\lambda_2)(t-s)}|E(s)|ds+
	\int_{t}^{\infty}e^{-(\lambda_3-\lambda_2)(t-s)}|E(s)|ds,
\\
\mathbb{F}_3(E)(t)
	&=&
	\int_{t_0}^{t}e^{-(\lambda_2-\lambda_3)(t-s)}|E(s)|ds+
	\int_{t}^{\infty}e^{-(\lambda_4-\lambda_3)(t-s)}|E(s)|ds,
\\
\mathbb{F}_4(E)(t)
	&=&
	\int_{t_0}^{t}e^{-(\lambda_3-\lambda_4)(t-s)}|E(s)|ds;
\end{eqnarray*}
the positive numbers $\upvarsigma_i,A_i$ defined by
\begin{eqnarray}
\upvarsigma_i&=&3|\lambda_i|^2+5|\lambda_i|+3
\nonumber\\
&&+
        \Big(19+7|\lambda_i|+|12\lambda_i+3a_3|+
        |6\lambda^2_i+3\lambda_ia_3+a_2|\Big)\upeta,
\quad \upeta\in ]0,1/2[,
\nonumber\\
A_i&=&
\frac{1}{|\Upsilon_i|}
\sum_{(j,k,\ell)\in I_i}
|\lambda_k-\lambda_{\ell}|
\Big(1+|\lambda_j-\lambda_i|+|\lambda_j-\lambda_i|^2\Big),
\label{eq:not_teor}
\end{eqnarray}
with
\begin{eqnarray*}
\Upsilon_i&=&\prod_{k>j}(\lambda_k-\lambda_j),
\quad
k,j\in \{1,2,3,4\}-\{i\},
\\
I_i&=&\Big\{\;
(j,k,\ell)\in\{1,2,3,4\}^3\quad:\quad (j,k,\ell)\not=(i,i,i),\; 
(k,\ell)\not=(j,j)
\;\Big\};
\end{eqnarray*}
and  define the sets 
\begin{eqnarray}
\mathcal{F}_{i}([t_0,\infty[)
	&=&\Bigg\{E:[t_0,\infty[\to\mathbb{R}
	\; :\;
	\mathbb{F}_{i}(E)(t)\le 
	\rho_i:=\min\left\{\mathbb{F}_i(1)(t),\frac{1}{A_{i}\upvarsigma_{i}}\right\}
	\;\;
	\Bigg\}.
\label{eq:setsFi}
\end{eqnarray}

\begin{theorem}
\label{teo:solution_asymptoticllll}
Consider that the hypotheses (H$_1$),(H$_2$)  
on Theorem~\ref{lem:solution_perturbation}
are satisfied. Moreover consider that the  following
hypothesis:
\begin{enumerate}
\item[(H$_3$)] 
Assume that the perturbation functions 
$r_0,r_1,r_2,r_3\in\mathcal{F}_i([t_0,\infty[)$. 
\end{enumerate}
is satisfied.
Then, $z_i$, the solution of \eqref{eq:ricati_original} with $\mu=\lambda_i$,
has the following asymptotic behavior 
\begin{eqnarray}
z_{i}(t),z'_{i}(t),z''_{i}(t)
=
\left\{
\begin{array}{lll}
\displaystyle
O\Big(\int_{t}^{\infty}e^{-\beta(t-s)}|p(\lambda_1,s)|ds\Big),
&\;
& i=1,\quad \beta\in ]\lambda_2-\lambda_1,0[,
\\
\displaystyle
O\Big(\int_{t_0}^{\infty}e^{-\beta(t-s)}|p(\lambda_2,s)|ds\Big),
&
& i=2,\quad \beta\in ]\lambda_3-\lambda_2,0[,
\\
\displaystyle
O\Big(\int_{t_0}^{\infty}e^{-\beta(t-s)}|p(\lambda_3,s)|ds\Big),
&
& i=3,\quad \beta\in ]\lambda_4-\lambda_3,0[,
\\
\displaystyle
O\Big(\int_{t_0}^{t}e^{-\beta(t-s)}|p(\lambda_4,s)|ds\Big),
&
& i=4,\quad \beta\in ]0,\lambda_3-\lambda_4[,
\end{array}
\right.
\label{eq:asymptotic_teo:formllll}
\end{eqnarray}
where $p$ is defined in \eqref{eq:termino_constante}.
\end{theorem}

\begin{proof}
By application of Theorem~\ref{teo:solution_asymptotic}.
\end{proof}

\begin{theorem}
\label{teo:levinson}
Let us assume that the hypotheses  
(H$_1$), (H$_2$) on Theorem~\ref{lem:solution_perturbation}
and hypothesis (H$_3$) on Theorem~\ref{teo:solution_asymptoticllll}
are satisfied.
Denote by
$\pi_i$ the number defined as follows
\begin{eqnarray*}
\pi_i= \prod_{k\in N_i}(\lambda_k-\lambda_i),
\quad 
N_i=\{1,2,3,4\}-\{i\},
\quad
i=1,\ldots,4,
\end{eqnarray*}
and by $p$ the function defined in \eqref{eq:termino_constante}.
Then, the following asymptotic behavior   
\begin{eqnarray}
y_i(t)&=& e^{\lambda_i(t-t_0)}
	\exp\Big(\pi^{-1}_i\int_{t_0}^{t}
	\Big[p(\lambda_i,s)+F(s,z_i(s),z'_i(s),z''_i(s))\Big]ds\Big),
\label{eq:asymptotic_perturbed_2}
\\
y'_i(t)&=&
\Big(\lambda_i+o(1)\Big)
e^{\lambda_i(t-t_0)}
	\exp\Big(\pi^{-1}_i\int_{t_0}^{t}
	\Big[p(\lambda_i,s)+F(s,z_i(s),z'_i(s),z''_i(s))\Big]ds\Big),
\label{eq:asymptotic_perturbed_3}
\\
y''_i(t)&=&
\Big(\lambda^2_i+o(1)\Big)
e^{\lambda_i(t-t_0)}
	\exp\Big(\pi^{-1}_i\int_{t_0}^{t}
	\Big[p(\lambda_i,s)+F(s,z_i(s),z'_i(s),z''_i(s))\Big]ds\Big),
\label{eq:asymptotic_perturbed_3_2}
\\
y'''_i(t)&=&
\Big(\lambda^3_i+o(1)\Big)
e^{\lambda_i(t-t_0)}
	\exp\Big(\pi^{-1}_i\int_{t_0}^{t}
	\Big[p(\lambda_i,s)+F(s,z_i(s),z'_i(s),z''_i(s))\Big]ds\Big),
\label{eq:asymptotic_perturbed_3_3}
\\
y^{({\rm iv})}_i(t)&=&
\Big(\lambda^4_i+o(1)\Big)
e^{\lambda_i(t-t_0)}
	\exp\Big(\pi^{-1}_i\int_{t_0}^{t}
	\Big[p(\lambda_i,s)+F(s,z_i(s),z'_i(s),z''_i(s))\Big]ds\Big),
\label{eq:asymptotic_perturbed_3_4}
\end{eqnarray}
holds, when $t\to\infty$ with 
$z_i,z'_i$ and $z''_i$ given asymptotically by 
\eqref{eq:asymptotic_teo:formllll}.
Moreover, if $r_0,r_1,r_2,r_3\in L^1([t_0,\infty[)$
then the asymptotic forms
\begin{eqnarray}
&&y_i(t)= e^{\lambda_i(t-t_0)}+o(1),
\hspace{1.5cm}
y'_i(t)=
\Big(\lambda_i+o(1)\Big)
e^{\lambda_i(t-t_0)},
\label{eq:asymptotic_levins_1}\\
&&
y''_i(t)=
\Big(\lambda^2_i+o(1)\Big)
e^{\lambda_i(t-t_0)},
\hspace{0.8cm}
y'''_i(t)=
\Big(\lambda^3_i+o(1)\Big)
e^{\lambda_i(t-t_0)},
\label{eq:asymptotic_levins_2}\\
&&
y^{({\rm iv})}_i(t)=
\Big(\lambda^4_i+o(1)\Big)
e^{\lambda_i(t-t_0)},
\label{eq:asymptotic_levins_3}
\end{eqnarray}
are satisfied when $t\to\infty$.
\end{theorem}

\begin{proof}
The proof of \eqref{eq:asymptotic_perturbed_2} follows from
the identity
\begin{eqnarray}
\int_{t_0}^t e^{-a\tau}\int_{\tau}^\infty e^{as} H(s)dsd\tau
&=&-\frac{1}{a}\left[
\int_{t}^\infty e^{-a(t-s)}H(s)ds-\int_{t_0}^\infty e^{-a(t_0-s)}H(s)ds
\right]
\nonumber\\
&&
\qquad
+\frac{1}{a}\int_{t_0}^t H(\tau)d\tau
\label{eq:idenaux_pro}
\end{eqnarray}
and from \eqref{eq:operator_fix_point}-\eqref{eq:operator_equation}. 
Now, we develop the proof for $i=1$.
Indeed,
by \eqref{eq:fundam_sist_sol} we have that
\begin{eqnarray}
y_1(t)=\exp\Big(\int_{t_0}^t(\lambda_1+z_1(\tau))d\tau\Big)
=e^{\lambda_1(t-t_0)}\exp\Big(\int_{t_0}^t z_1(\tau)d\tau\Big).
\label{eq:solfun_paso1}
\end{eqnarray}
By \eqref{eq:operator_fix_point}-\eqref{eq:operator_equation},
\eqref{eq:operator_fix_point_g1}, \eqref{eq:lambda_vs_gamma:1}, 
\eqref{eq:idenaux_pro},
and the fact that 
$\pi_1=(\lambda_2-\lambda_1)(\lambda_3-\lambda_1)(\lambda_4-\lambda_1),$
we have that
\begin{eqnarray*}
\int_{t_0}^t z_1(\tau)d\tau
&=& \frac{1}{\Upsilon_1}\int_{t_0}^t \int_{t_0}^\infty g_1(\tau,s)
	\Big(p(\lambda_1,s)+F(s,z_1(s),z'_1(s),z''_1(s)\Big)d\tau ds
\\
&=& \frac{1}{\Upsilon_1}\int_{t_0}^t \int_{\tau}^\infty 
	\Big[
	(\lambda_4-\lambda_3)e^{-(\lambda_2-\lambda_1)(\tau-s)}
	+(\lambda_2-\lambda_4)e^{-(\lambda_3-\lambda_1)(\tau-s)}
\\
&&\hspace{1.8cm}+(\lambda_3-\lambda_2)e^{-(\lambda_4-\lambda_1)(\tau-s)}\Big]
	\Big(p(\lambda_1,s)+F(s,z_1(s),z'_1(s),z''_1(s))\Big)d\tau ds
\\
&=&
\frac{1}{\Upsilon_1}
	\left[
	\frac{\lambda_4-\lambda_3}{\lambda_2-\lambda_1}
	+\frac{\lambda_2-\lambda_4}{\lambda_3-\lambda_1}
	+\frac{\lambda_3-\lambda_2}{\lambda_4-\lambda_1} 
	\right]
	\int_{t_0}^t\Big(p(\lambda_1,s)+F(s,z_1(s),z'_1(s),z''_1(s))\Big) ds
\\
&&
+\frac{1}{\Upsilon_1}
	\Bigg[\frac{\lambda_4-\lambda_2}{\lambda_2-\lambda_1}
	\Bigg\{
	\int_{t}^\infty 
	e^{-(\lambda_2-\lambda_1)(t-s)}
	\Big(p(\lambda_1,s)+F(s,z_1(s),z'_1(s),z''_1(s))\Big)ds
\\
&&
\hspace{2.5cm}
	-\int_{t_0}^\infty 
	e^{-(\lambda_2-\lambda_1)(t_0-s)}
	\Big(p(\lambda_1,s)+F(s,z_1(s),z'_1(s),z''_1(s))\Big)ds
	\Bigg\}
	\Bigg]
\\
&&
+\frac{1}{\Upsilon_1}
	\Bigg[\frac{\lambda_2-\lambda_4}{\lambda_3-\lambda_1}
	\Bigg\{
	\int_{t}^\infty 
	e^{-(\lambda_3-\lambda_1)(t-s)}
	\Big(p(\lambda_1,s)+F(s,z_1(s),z'_1(s),z''_1(s))\Big)ds
\\
&&
\hspace{2.5cm}
	-\int_{t_0}^\infty 
	e^{-(\lambda_3-\lambda_1)(t_0-s)}
	\Big(p(\lambda_1,s)+F(s,z_1(s),z'_1(s),z''_1(s))\Big)ds
	\Bigg\}
	\Bigg]
\\
&&
+\frac{1}{\Upsilon_1}
	\Bigg[\frac{\lambda_3-\lambda_2}{\lambda_4-\lambda_1}
	\Bigg\{
	\int_{t}^\infty 
	e^{-(\lambda_4-\lambda_1)(t-s)}
	\Big(p(\lambda_1,s)+F(s,z_1(s),z'_1(s),z''_1(s))\Big)ds
\\
&&
\hspace{2.5cm}
	-\int_{t_0}^\infty 
	e^{-(\lambda_4-\lambda_1)(t_0-s)}
	\Big(p(\lambda_1,s)+F(s,z_1(s),z'_1(s),z''_1(s))\Big)ds
	\Bigg\}
	\Bigg]
\\
&=&
\frac{1}{\pi_1}
	\int_{t_0}^t\Big(p(\lambda_1,s)+F(s,z_1(s),z'_1(s),z''_1(s))\Big)ds
	+o(1)
\end{eqnarray*}
Then, \eqref{eq:asymptotic_perturbed_2} is valid for $i=1$.
The proof of \eqref{eq:asymptotic_perturbed_2} for $i=2,3,4$ is analogous.
Now the proof of  \eqref{eq:asymptotic_perturbed_3} follows by
\eqref{eq:asymptotic_perturbed_2} and \eqref{eq:deri_uno}-\eqref{eq:deri_cuatro}.

To prove \eqref{eq:asymptotic_levins_1}-\eqref{eq:asymptotic_levins_3}
we apply the decomposition of Theorem~\ref{teo:solution_integrab}.
\end{proof}

\subsection{Hartman-Wintner and Harris-Lutz  type theorems.} $ $

By application of Theorem~\ref{teo:solution_integrab}
we get the following results.

\begin{theorem}
\label{teo:H-W}
(Hartman-Wintner)
Let us assume that the hypotheses  
(H$_1$), (H$_2$) on Theorem~\ref{lem:solution_perturbation}
and hypothesis (H$_3$) on Theorem~\ref{teo:solution_asymptoticllll}
are satisfied. If if $r_0,r_1,r_2,r_3\in L^1([t_0,\infty[)$
for $p\in ]1,2]$, then 
the asymptotic behavior
\eqref{eq:asymptotic_levins_1}-\eqref{eq:asymptotic_levins_3}
is valid when $t\to\infty$.
\end{theorem}

\begin{theorem}
\label{teo:H-L}
(Harris-Lutz)
Let us assume that the hypotheses  
(H$_1$), (H$_2$) on Theorem~\ref{lem:solution_perturbation}
and hypothesis (H$_3$) on Theorem~\ref{teo:solution_asymptoticllll}
are satisfied. If $r_0,r_1,r_2,r_3\in L^p([t_0,\infty[)$
for $p\ge 1$, then the following 
asymptotic behavior   
\begin{eqnarray}
y_i(t)&=& e^{\lambda_i(t-t_0)}
	\exp\Big(\int_{t_0}^{t}
	\Big[\sum_{k=1}^m\Theta_{k}(s)+\Psi(s)\Big]ds\Big),
\label{eq:asymptotic_perron_0}
\\
y'_i(t)&=&
\Big(\lambda_i+o(1)\Big)
e^{\lambda_i(t-t_0)}
	\exp\Big(\int_{t_0}^{t}
	\Big[\sum_{k=1}^m\Theta'_{k}(s)+\Psi'(s)\Big]ds\Big),
\label{eq:asymptotic_perron_1}
\\
y''_i(t)&=&
\Big(\lambda^2_i+o(1)\Big)
e^{\lambda_i(t-t_0)}
	\exp\Big(\int_{t_0}^{t}
	\Big[\sum_{k=1}^m\Theta''_{k}(s)+\Psi''(s)\Big]ds\Big),
\label{eq:asymptotic_perron_2}
\\
y'''_i(t)&=&
\Big(\lambda^3_i+o(1)\Big)
e^{\lambda_i(t-t_0)}
	\exp\Big(\int_{t_0}^{t}
	\Big[\sum_{k=1}^m\Theta'''_{k}(s)+\Psi'''(s)\Big]ds\Big),
\label{eq:asymptotic_perron_3}
\\
y^{({\rm iv})}_i(t)&=&
\Big(\lambda^4_i+o(1)\Big)
e^{\lambda_i(t-t_0)}
	\exp\Big(\int_{t_0}^{t}
	\Big[\sum_{k=1}^m\Theta^{({\rm iv})}_{k}(s)+\Psi^{({\rm iv})}(s)\Big]ds\Big),
\label{eq:asymptotic_perron_4}
\end{eqnarray}
holds, when $t\to\infty$,
where $m,\Theta_i,\psi$ is the notation
introduced on Theorem~\ref{teo:solution_integrab}.
\end{theorem}


\subsection{Unbounded coefficients}
Let us consider the following equation
\begin{eqnarray}
y^{({\rm iv})}(t)
-2[q(t)]^{1/2} y'''(t)
-q(t) y''(t)
+2[q(t)]^{3/2} y'(t)
+r(t) y(t)=0,
\label{eq:unbounded:1}
\end{eqnarray}
where $q$ and $r$ are given functions such that
\begin{eqnarray}
q(t)\to\infty,
\qquad
\mbox{when}
\qquad
t\to\infty.
\label{eq:unbounded:2}
\end{eqnarray}
Then, we note that \eqref{eq:unbounded:1} is not of
a Poincar\'e type. However, we can apply the results
of asymptotic behavior Theorems~\ref{teo:poincare}, \ref{teo:solution_asymptotic}
and \ref{teo:levinson}, after an appropriate
transformation. Indeed, let us consider the following
change of variable
\begin{eqnarray}
z(s)=y(s)[q(s)]^{1/4}
\qquad
\mbox{with}
\qquad
s=\int_{t_0}^t[q(\tau)]^{1/2}d\tau.
\end{eqnarray}
Note that $ds(t)=[q(t)]^{1/2}dt$. Then, we can rewrite \eqref{eq:unbounded:1} 
in an equivalent way as follows
\begin{eqnarray}
&&z^{({\rm iv})}
+\left(-2+r_3(t)\right) z'''
+\left(-1+r_2(t)\right) z''
+\left(2+r_1(t)\right) z'
+r_0(t) z=0,
\end{eqnarray}
where
\begin{eqnarray}
r_0(t)&=&
-\frac{q^{({\rm iv})}(t)}{4q(t)}+\frac{q'''(t)}{2q(t)}
-\frac{5q'''(t)}{8[q(t)]^{5/2}}
+\frac{5q'''(t)q'(t)}{4[q(t)]^{2}}
+\left(\frac{3}{8[q(t)]^3}+\frac{15}{16[q(t)]^2}\right)[q''(t)]^2
\nonumber\\
&&-\left(\frac{15}{8[q(t)]^2}-\frac{1}{8[q(t)]^{7/2}}-\frac{1}{4[q(t)]^3}\right)q''(t)q'(t)
+\frac{q''(t)}{4q(t)}
\nonumber\\
&&
-\left(\frac{135}{32[q(t)]^3}-\frac{5}{4[q(t)]^{7/2}}-\frac{11}{16[q(t)]^4}
-\frac{3}{16[q(t)]^{9/2}}\right)q''(t)[q'(t)]^2-\frac{q'(t)}{2q(t)}
\nonumber\\
&&
+\left(\frac{1}{8[q(t)]^{5/2}}-\frac{5}{16[q(t)]^2}\right)[q'(t)]^2
+\left(\frac{45}{32[q(t)]^{3}}-\frac{15}{16[q(t)]^{7/2}}-\frac{1}{8[q(t)]^4}\right)[q'(t)]^3
\nonumber\\
&&
+\left(\frac{585}{256[q(t)]^{4}}-\frac{135}{64[q(t)]^{9/2}}-\frac{5}{64[q(t)]^5}-\frac{3}{32[q(t)]^{11/2}}\right)[q'(t)]^4
+\frac{r(t)}{[q(t)]^2},
\label{eq:unb:r_0}\\
r_1(t)&=&\left(\frac{1}{2[q(t)]^{5/2}}-\frac{1}{q(t)}\right)q'''(t)
+\left(\frac{3}{2q(t)}-\frac{1}{[q(t)]^{2}}\right)q''(t)
\nonumber\\
&&
+\left(\frac{15}{4[q(t)]^{2}}-\frac{17}{8[q(t)]^{5/2}}
-\frac{1}{[q(t)]^{3}}-\frac{3}{4[q(t)]^{7/2}}\right)q''(t)q'(t)
\nonumber\\
&&
-\left(\frac{45}{16[q(t)]^{3}}+\frac{45}{16[q(t)]^{7/2}}-\frac{3}{8[q(t)]^{9/2}}
-\frac{1}{8[q(t)]^{4}}\right)[q'(t)]^3
\nonumber\\
&&
-\left(\frac{15}{8[q(t)]^{2}}-\frac{3}{2[q(t)]^{5/2}}-\frac{1}{2[q(t)]^{3}}\right)[q'(t)]^2
+\left(\frac{1}{2q(t)}-\frac{1}{2[q(t)]^{3/2}}\right)q'(t)
\nonumber\\
&&-\frac{17}{8[q(t)]^{5/2}} ,
\label{eq:unb:r_1}\\
r_2(t)&=&\left(\frac{1}{2[q(t)]^4}+\frac{3}{2[q(t)]^2}-\frac{3}{2q(t)}\right)q''(t)
\nonumber\\
&&
+\left(\frac{15}{16[q(t)]^2}-\frac{3}{4[q(t)]^{5/2}}-\frac{1}{4[q(t)]^{3}}
-\frac{3}{8[q(t)]^{4}}\right)[q'(t)]^2
\nonumber\\
&&
-\left(\frac{9}{8[q(t)]^{5/2}}+\frac{15}{16[q(t)]^2}+\frac{3}{[q(t)]^{3/2}}
-\frac{3}{2q(t)}\right)q'(t),
\label{eq:unb:r_2}\\
r_3(t)&=&\left(\frac{3}{[q(t)]^{3/2}}-\frac{1}{4q(t)}\right)q'(t)-\frac{3}{4q(t)}.
\label{eq:unb:r_3}
\end{eqnarray}
Thus, by application of Theorem~\ref{teo:solution_asymptoticllll} we get the following 
theorem.

\begin{theorem}
\label{eq:teorem_unbounded}
Let us consider that $q$ and $r$ are two functions such that the functions
\begin{eqnarray*}
&&
\left(\frac{q'}{q}\right)^{2k}\mbox{ for } k=1,\ldots,4;
\quad
\left(\frac{q''}{q}\right)^{2k}\mbox{ for } k=1,2;
\qquad
\frac{(q'')^2(q')^{2(k+1)}}{q^{2(k+2)}}\mbox{ for } k=0,1;
\\
&&
\frac{r^2}{q^4};\quad \frac{q'''q'}{q^2};\quad \left(\frac{q'''}{q}\right)^2\quad
\mbox{and}\quad
 \left(\frac{q^{(iv)}}{q}\right)^2
\end{eqnarray*}
are in $L^1([t_0,\infty))$. Moreover assume that
$q\in C^4([t_0,\infty))$ is an increasing function on
$[t_0,\infty)$  and satisfies \eqref{eq:unbounded:2}.
Then the equation \eqref{eq:unbounded:1} asymptotically has
the following fundamental system of solutions
\begin{eqnarray}
y_i(t)&=&[q(t)]^{-1/4}\big(1+o(1)\big)
\nonumber\\
&&
\times
\left\{
\begin{array}{lll}
 \exp\Big(-\int_{t_0}^t[q(\tau)]^{1/2}d\tau
 +\frac{1}{6}\int_{t_0}^t(r_0-r_1+r_2-r_3)(\tau)[q(\tau)]^{1/2}d\tau\Big),
 &\quad& i=1,
 \\
  \exp\Big(-\frac{1}{2}\int_{t_0}^t[q(\tau)]^{1/2}r_0(\tau)d\tau\Big),
 &\quad& i=2,
 \\
  \exp\Big(\int_{t_0}^t[q(\tau)]^{1/2}d\tau
 +\frac{1}{2}\int_{t_0}^t(r_0+r_1+r_2+r_3)(\tau)[q(\tau)]^{1/2}d\tau\Big),
 &\quad& i=3,
 \\
  \exp\Big(2\int_{t_0}^t[q(\tau)]^{1/2}d\tau
 -\frac{1}{6}\int_{t_0}^t(r_0+2r_1+4r_2+8r_3)(\tau)[q(\tau)]^{1/2}d\tau\Big),
 &\quad& i=4,
\end{array}
\right.
\label{eq:asim_for:unbound}
\end{eqnarray}
where $r_i$ are defined on \eqref{eq:unb:r_0}-\eqref{eq:unb:r_3}.

\end{theorem}

\section{Examples}
\label{sec:examples}

In this section we consider three examples. In the first two examples the classical
results of Levinson, Hartman-Witner and Harris-Lutz 
cannot be applied but we can apply the Theorem~\ref{teo:levinson}. In the 
third example we present an application of Theorem~\ref{eq:teorem_unbounded}.

\subsection{Example 1} 
In this example we present a case 
where the results Levinson type or Hartman-Witner type or Harris-Lutz type cannot be applied.
Indeed, let us consider the differential equation
\begin{eqnarray}
y^{(4)}+2y^{(3)}+13y^{(2)}-14y^{(1)}+\left [
\frac{3}{t^{1/(p+1)}(\sin t +2)}+24 \right ] y=0,\quad p\geq
1, \quad t\in[1,+\infty[.
\label{eq:example1}
\end{eqnarray}
We note that $r_{0}(t)=3t^{-1/p+1}\;(\sin t +2)^{-1}\notin L^{p}([1,+\infty[) $
for any $p\geq 1$.
Then, the classical generalizations of Poincar\'{e} type theorems:
the Levinson and the Hartman--Wintner  theorems, can not be applied to
obtain the asymptotic behavior of \eqref{eq:example1}.
Now, in order to apply the Theorem~\ref{teo:levinson}, we have that
\begin{enumerate}
 \item[(a)] The set of characteristic roots of the linear 
 no perturbed equation associated to (4.1), i.e.
 $y^{(4)}+2y^{(3)}+13y^{(2)}-14y^{(1)}+24y=0,$ is given 
 by $\{3,1,-2,-4\}$. Then (H1) is satisfied.
 
 \item[(b)] We note that $r_{0}(t)\to 0$ when $t \to \infty$
 which implies that $\mathfrak{L}(r_{0})(t)\to 0$ when $t \to \infty$
 and $\mathcal{L}(r_{1})(t)=\mathcal{L}(r_{2})(t)=\mathcal{L}(r_{0})(t)=0$ since
 $r_{1}=r_{2}=r_{3}=0$. Moreover $p(\lambda_{i},s)=(r_{0})(s)$ 
 and as $\mathcal{g}(p(\lambda_{i},.))(t)=\mathcal{G}(r_{0})(t))\leq
 \mathcal{L}(r_{0})(t)$ we deduced that $\mathcal{G}(p(\lambda_{i},.))(t)\to
 0$ when $t \to \infty$. Thus, (H2) is also satisfied.
 
\item[(c)] We note that
\begin{eqnarray*}
&&\mathbb{F}_{1}=\frac{1}{2},\quad
\mathbb{F}_{2}=\frac{1}{2}[1-e^{-2(t-1)}]+\frac{1}{3}, \quad
\mathbb{F}_{3}=\frac{1}{3}[1-e^{-3(t-1)}]+\frac{1}{2}, \quad
\mathbb{F}_{4}=\frac{1}{2}[1-e^{-2(t-1)}],
\\
&&\sigma_{1}=45+167\eta,\quad 
\sigma_{2}=11+69\eta, \quad
\sigma_{3}=25+76\eta,\quad \sigma_{4}=61+174\eta,
\\
&&
A_{1}=\frac{34}{3},\quad
A_{2}=\frac{26}{7},\quad
A_{3}=\frac{26}{7},\quad
A_{4}=\frac{34}{3}.
\end{eqnarray*}
Then, the sets $\mathcal{F}_{i}([1,\infty[)$ given in given (3.19)
are well defined. We note that, naturally,$r_{1}=r_{2}=r_{3}=0\in
\mathcal{F}_{i}$. Moreover, from $r_{0}(t)\to 0$ when $t \to \infty$
we can prove that $\mathbb{F}_{i}(r_{0}(t))\to 0$ when $t \to
\infty$. Then, we have that $r_{0}\in \mathcal{F}_{i}([1,\infty[)$.
Hence, (H3) is satisfied.

\end{enumerate}
Thus, from $(a)-(c)$, we can apply the Theorem~\ref{teo:levinson} and the
asymptotic formulas are given by
\begin{eqnarray*}
y_1(t)
    &=&
    e^{3(t-s)}\exp\left \{\left (\frac{1}{30}\int_{1}^{t}\left 
    [\frac{6}{s^{\frac{1}{p+1}}[\sin s +2]}+f_{1}(s)\right ]ds \right ) \right \},
\\
y_2(t)
    &=&
    e^{(t-s)}\exp\left \{\left (\frac{-1}{70}\int_{1}^{t}\left 
    [\frac{6}{s^{\frac{1}{p+1}}[\sin s +2]}+f_{2}(s)\right ]ds \right ) \right \},
\\
y_3(t)
    &=&
    e^{-2(t-s)}\exp\left \{\left (\frac{1}{70}\int_{1}^{t}\left 
    [\frac{6}{s^{\frac{1}{p+1}}[\sin s +2]}-f_{3}(s)\right ]ds \right ) \right \},
\\
y_4(t)
    &=&
    e^{-4(t-s)}\exp\left \{\left (\frac{-1}{30}\int_{1}^{t}\left 
    [\frac{6}{s^{\frac{1}{p+1}}[\sin s +2]}+f_{4}(s)\right ]ds \right ) \right \},
\end{eqnarray*}
where
\begin{eqnarray*}
f_1(t)
    &=&
    -42z_{1}z_{1}^{'}+4z_{1}z_{1}^{''}+265z_{1}^{2}
    +3(z_{1}')^{2}+6z_{1}^{2}z_{1}^{'}+24z_{1}^{3}+z_{1}^{4},
\\
f_2(t)
    &=&
    -18z_{2}z_{2}^{'}+4z_{2}z_{2}^{''}+25z_{2}^{2}
    +3(z_{2}')^{2}+6z_{2}^{2}z_{2}^{'}+4z_{2}^{3}+z_{2}^{4},
\\
f_3(t)
    &=&
    -18z_{3}z_{3}^{'}+4z_{3}z_{3}^{''}+25z_{3}^{2}
    +3(z_{3}')^{2}+6z_{3}^{2}z_{3}^{'}-8z_{3}^{3}+z_{3}^{4},
\\
f_{4}(t)
    &=&
    -42z_{4}z_{4}^{'}+4z_{4}z_{4}^{''}+84z_{4}^{2}
    +3(z_{4}')^{2}+6z_{4}^{2}z_{4}^{'}-16z_{4}^{3}+z_{4}^{4},
\end{eqnarray*}
and $z_{i}(t)$ satisfies the following asymptotic behavior
\begin{eqnarray*}
z_{i}(t),z_{i}'(t), z_{i}''(t)= \left\{
\begin{array}{lll}
\displaystyle
O\Big(\int_{t}^{\infty}\frac{3e^{-\beta(t-s)}}{s^{\frac{1}{p+1}}[\sin
s +2]}ds\Big), &\qquad & i=1,\quad \beta\in [-2,0[,
\\
\displaystyle
O\Big(\int_{1}^{\infty}\frac{3e^{-\beta(t-s)}}{s^{\frac{1}{p+1}}[\sin
s +2]}ds\Big), &\qquad & i=2,\quad \beta\in [-3,0[,
\\
\displaystyle
O\Big(\int_{1}^{\infty}\frac{3e^{-\beta(t-s)}}{s^{\frac{1}{p+1}}[\sin
s +2]}ds\Big), &\qquad & i=3,\quad \beta\in [-2,0[,
\\
\displaystyle
O\Big(\int_{1}^{t}\frac{3e^{-\beta(t-s)}}{s^{\frac{1}{p+1}}[\sin s
+2]}ds\Big), &\qquad & i=4,\quad \beta\in ]0,7],
\end{array}
\right.
\end{eqnarray*}

\subsection{Example 2}
Here we present an example with  a perturbation function such that
the results of Levinson, Hartman-Witner, Harris-Lutz 
or Eastham types cannot be applied. Indeed, let us consider 
the differential equation
\begin{eqnarray}
y^{(4)}+10y^{(3)}+35y^{(2)}+50y^{(1)}+\left [ \frac{3}{(\cos t
+2)\log t}+24 \right ] y=0,\quad t\in[2,+\infty[.
\label{eq:example:2}
\end{eqnarray}
We note that
\begin{eqnarray*}
r_{0}(t)=\left [ \frac{3}{\log t[\cos t +2]}\right ]\notin
L^{p}([2,+\infty[),\text{ for any}\quad p\geq 1
\end{eqnarray*}
Then, the classical generalizations of Poincar\'{e} type theorems,
the Levinson theorem, the Hartman--Wintner  can not be applied to
obtain the asymptotic behavior of the solutions for \eqref{eq:example:2}.
Moreover, we note that the inequality
\begin{eqnarray*}
\frac{|\sin t|}{9\log t}\leq \frac{|\sin t|}{(\cos t+2)^{2}\log t}
\end{eqnarray*}
is valid on $[2,+\infty[$, which implies, by the comparison criteria and since 
$\sin t/\log t[\cos t+2]^{2}\notin L^{1}([2,+\infty[),$ the fact that 
\begin{eqnarray*}
\frac{\sin t}{\log t[\cos t+2]^{2}} \notin L^{1}([2,+\infty[).
\end{eqnarray*}
Then,  we can deduce that
\begin{eqnarray*}
r_{0}'(t)=3\Big[\frac{1}{ t(\log t)^{2}[\cos t+2]\ln 10}+\frac{\sin
t}{\log t[\cos t+2]^{2}}\Big]\notin L^{1}([2,+\infty[).
\end{eqnarray*}
Thus, we can not apply the classic theorem of Eastham.
However, we note that
\begin{enumerate}
 \item[(a)] The set of characteristic roots of the linear
 no perturbed equation associated to \eqref{eq:example:2}, i.e.
 $y^{(4)}+10y^{(3)}+35y^{(2)}+50y^{(1)}+24y=0,$ is given 
 by $\{-1,-2,-3,-4\}$. Then (H1) is satisfied.
 
 \item[(b)] We note that $r_{0}(t)\to 0$ when
 $t \to \infty$ which implies that $\mathcal{L}(r_{0})(t)\to 0$ 
 when $t \to \infty$
 and $\mathcal{L}(r_{1})(t)=\mathfrak{L}(r_{2})(t)=\mathcal{L}(r_{0})(t)=0$ since
 $r_{1}=r_{2}=r_{3}=0$.
 Moreover $p(\lambda_{i},s)=(r_{0})(s)$ and 
 as $\mathcal{G}(p(\lambda_{i},.))(t)=\mathcal{G}(r_{0})(t))\leq
 \mathcal{L}(r_{0})(t)$ we deduced 
 that $\mathcal{G}(p(\lambda_{i},.))(t)\to
 0$ when $t \to \infty$. Thus, 
 (H2) is also satisfied.
 
\item[(c)] We note that
\begin{eqnarray*}
&&\mathbb{F}_{1}=1,\quad
\mathbb{F}_{2}=2-e^{(t-2)},\quad
\mathbb{F}_{3}=2-e^{-(t-2)},\quad
\mathbb{F}_{4}=1-e^{-(t-2)},
\\
&&\sigma_{1}=11+55\eta,\quad
\sigma_{2}=25+38\eta,\quad
\sigma_{3}=45+47\eta,\quad
\sigma_{4}=71+76\eta,
\\
&&
A_{1}=\frac{29}{2},\quad
A_{2}=\frac{13}{3},\quad
A_{3}=\frac{13}{3},\quad
A_{4}=\frac{29}{2}.
\end{eqnarray*}
Then, the sets $\mathcal{F}_{i}([1,\infty[)$ given in given (3.19)
are well defined. We note that, naturally, $r_{1}=r_{2}=r_{3}=0\in
\mathcal{F}_{i}$. Moreover, from $r_{0}(t)\to 0$ when $t \to \infty$
we can prove that $\mathbb{F}_{i}(r_{0}(t))\to 0$ when $t \to
\infty$. Then, we have that $r_{0}\in \mathcal{F}_{i}([1,\infty[)$.
Hence, (H3) is satisfied.
\end{enumerate}
Thus, from $(a)-(c)$, we can apply the Theorem~\ref{teo:levinson} and the
asymptotic formulas are given by
\begin{eqnarray*}
y_1(t)
    &=&
    e^{-(t-s)}\exp\left \{\left (\frac{1}{6}\int_{1}^{t}
    \left [\frac{6}{(\cos s +2)\log s}+f_{1}(s)\right ]ds \right ) \right \},
\\
y_2(t)
    &=&
    e^{-2(t-s)}\exp\left \{\left (\frac{-1}{2}\int_{1}^{t}\left 
    [\frac{6}{(\cos s +2)\log s}+f_{2}(s)\right ]ds \right ) \right \},
\\
y_3(t)
    &=&
    e^{-3(t-s)}\exp\left \{\left (\frac{1}{2}\int_{1}^{t}\left 
    [\frac{6}{(\cos s +2)\log s}-f_{3}(s)\right ]ds \right ) \right \},
\\
y_4(t)
    &=&
    e^{-4(t-s)}\exp\left \{\left (\frac{-1}{6}\int_{1}^{t}\left 
    [\frac{6}{(\cos s +2)\log s}-f_{4}(s)\right ]ds \right ) \right \},
\end{eqnarray*}
where
\begin{eqnarray*}
f_1(t)
    &=&
    18z_{1}z_{1}^{'}+4z_{1}z_{1}^{''}+6z_{1}^{2}+3(z_{1}')^{2}
    +6z_{1}^{2}z_{1}^{'}-4z_{1}^{3}+z_{1}^{4},
\\
f_2(t)
    &=&
    6z_{2}z_{2}^{'}+4z_{2}z_{2}^{''}-z_{2}^{2}+3(z_{2}')^{2}
    +6z_{2}^{2}z_{2}^{'}-8z_{2}^{3}+z_{2}^{4},
\\
f_3(t)
    &=&
    -6z_{3}z_{3}^{'}+4z_{3}z_{3}^{''}-z_{3}^{2}+3(z_{3}')^{2}
    +6z_{3}^{2}z_{3}^{'}-12z_{3}^{3}+z_{3}^{4},
\\
f_{4}(t)
    &=&
    -18z_{4}z_{4}^{'}+4z_{4}z_{4}^{''}+11z_{4}^{2}
    +3(z_{4}')^{2}+6z_{4}^{2}z_{4}^{'}-16z_{4}^{3}+z_{4}^{4},
\end{eqnarray*}
and $z_{i}(t)$ satisfies the following asymptotic behavior
\begin{eqnarray*}
z_{i}(t),z_{i}'(t), z_{i}''(t)= \left\{
\begin{array}{lll}
\displaystyle O\Big(\int_{t}^{\infty}\frac{3e^{-\beta(t-s)}}
{(\cos s +2)\log s}ds\Big), &\qquad & i=1,\quad \beta\in [-1,0[,
\\
\displaystyle O\Big(\int_{1}^{\infty}\frac{3e^{-\beta(t-s)}}
{(\cos s +2)\log s}ds\Big), &\qquad & i=2,\quad \beta\in [-1,0[,
\\
\displaystyle O\Big(\int_{1}^{\infty}\frac{3e^{-\beta(t-s)}}
{(\cos s +2)\log s}ds\Big), &\qquad & i=3,\quad \beta\in [-1,0[,
\\
\displaystyle O\Big(\int_{1}^{t}\frac{3e^{-\beta(t-s)}}
{(\cos s +2)\log s}ds\Big), &\qquad & i=4,\quad \beta\in ]0,1].
\end{array}
\right.
\end{eqnarray*}

\subsection{Example 3}
Let us consider the following equation
\begin{eqnarray}
y^{({\rm iv})}(t)
-2t^{\alpha/2} y'''(t)
-t^{\alpha} y''(t)
+2t^{\alpha/2} y'(t)
+ y(t)=0,
\quad
\alpha >0,
\label{eq:unbounded:example}
\end{eqnarray}
i.e. the equation~\eqref{eq:unbounded:1} with $q(t)=t^\alpha$ and $r(t)=1$. 
The coefficients are unbounded and the classical results for the  asymptotic
behavior cannot be applied. However,
we note that 
\begin{eqnarray*}
&&
\left(\frac{q'(t)}{q(t)}\right)^{2k}
=\left(\frac{\alpha}{t}\right)^{2k}\mbox{ for } k=1,\ldots,4;
\quad 
\left(\frac{q''(t)}{q(t)}\right)^{2k}
=\left(\frac{\alpha(\alpha-1)}{t^2}\right)^{2k}\mbox{ for } k=1,2;
\\
&&
\frac{(q''(t))^2(q'(t))^{2(k+1)}}{q(t)^{2(k+2)}}
=\frac{\alpha^{2k+4}(\alpha-1)^2}{t^{2(k+3)}}
\mbox{ for } k=0,1;
\quad \frac{r(t)^2}{q(t)^4}=\frac{1}{t^{4\alpha}};
\\
&&
\frac{q'''(t)q'(t)}{q(t)^2}=\frac{\alpha(\alpha-1)(\alpha-2)}{t^4};
\quad \left(\frac{q'''(t)}{q(t)}\right)^2=\frac{\alpha^2(\alpha-1)(\alpha-2)^2}{t^6}
\quad\mbox{and}
\\
&&
 \left(\frac{q^{(iv)}}{q}\right)^2=\frac{\alpha^2(\alpha-1)^2(\alpha-2)^2(\alpha-3)^2}{t^8}.
\end{eqnarray*}
Then the hypotheses of Theorem~\ref{eq:teorem_unbounded} are satisfied and
the equation \eqref{eq:unbounded:example} has a fundamental system of solutions given
by \eqref{eq:asim_for:unbound} with $q(t)=t^\alpha$.

\subsection*{Acknowledgments}

An{\'\i}bal Coronel and Fernando Huancas
would like to thank the  support of research projects DIUBB 153209 GI/C and 
DIUBB 153109 G/EF at Universidad del B{\'\i}o-B{\'\i}o, Chile.
Manuel Pinto thanks for the  support of Fondecyt project~1120709.


\begin{thebibliography}{00}

\bibitem{aftabizadeh_1986}
A. R. Aftabizadeh; Existence and uniqueness theorems for fourth-order boundary value problems,
{\em Journal of Mathematical Analysis and Applications}, {\bf 116(2)} (1986), 415--426.


\bibitem{bellman_1949}
R. Bellman;
 {\em A Survey of the Theory of the Boundedness, Stability, and
  Asymptotic Behavior of Solutions of Linear and Non-linear Differential and
  Difference Equations}.
  Office of Naval Research of United States, Department of the Navy,
  NAVEXOS P-596,
  1949.

\bibitem{bellman_1950}
R. Bellman;
 On the asymptotic behavior of solutions of $u{''}-(1+f(t))u=0$.
 {\em Annali di Matematica Pura ed Applicata}, {\bf 31(1)} (1950), 83--91.

\bibitem{bellman_book}
R. Bellman;
 {\em Stability Theory of Differential Equations}.
 McGraw-Hill Book Company, Inc., New York-Toronto-London, 1953. 
 
 
\bibitem{codilevi_book}
E.A. Coddington, N. Levinson;
 {\em Theory of Ordinary Differential Equations}.
 McGraw-Hill Book Company, Inc., New York-Toronto-London, 1955.

\bibitem{coppel_book}
W.~A. Coppel;
 {\em Stability and Asymptotic Behavior of Differential Equations}.
 D. C. Heath and Co., Boston, Mass., 1965. 

 \bibitem{coronel_ejqde_2015}
A. Coronel, F. Huancas
and M. Pinto;
 {\em Asymptotic integration of a 
 linear fourth order differential equation 
 of Poincar\'e type,}
 Electron. J. Qual. Theory Differ. Equ., {\bf 76}
 ( 2015) 1--24.  
 
\bibitem{davies_1988}
A. R. Davies, A. Karageorghis, T. N. Phillips; Spectral Galerkin methods for
the primary two-point boundary value problem in modelling viscoelastic flows,
{\em Int. J. Numer. Methods Engng}, {\bf 26} (1988), 647--662 
 
\bibitem{eastham_portugalia}
M.S.P. Eastham;
 The asymptotic solution of higher-order differential equations with
  small final coefficient.
 {\em Portugal. Math.}, {\bf 45(4)} (1988), 351--362.
 
\bibitem{eastham_book}
M.S.P. Eastham;
 {\em The Asymptotic Solution of Linear Differential Dystems,
 Applications of the Levinson theorem}.
  London Mathematical Society Monographs, volume~4, 
  Oxford University Press, New York, 1989.

\bibitem{uri_harry}
U. Elias, H. Gingold;
 A framework for asymptotic integration of differential systems.
 {\em Asymptot. Anal.}, {\bf 35(3-4)} (2003), 281--300.  

\bibitem{fedoryuk_book}
M.V. Fedoryuk;
 {\em Asymptotic Analysis:Linear ordinary differential equations}
 (Translated from the Russian by Andrew Rodick).
 Springer-Verlag, Berlin, 1993.  

\bibitem{figueroa_2006}
P. Figueroa, M. Pinto;
 Asymptotic expansion of the variable eigenvalue associated to second
  order differential equations.
 {\em Nonlinear Stud.}, {\bf 13(3)} (2006), 261--272.
 
\bibitem{figueroa_2008}
P. Figueroa, M. Pinto;
 Riccati equations and nonoscillatory solutions of third order
  differential equations.
 {\em Dynam. Systems Appl.}, {\bf 17(3-4)} (2008), 459--475.
 
\bibitem{figueroa_2010}
P. Figueroa, M. Pinto;
 {$L^p$}-solutions of {R}iccati-type differential equations and
  asymptotics of third order linear differential equations.
 {\em Dyn. Contin. Discrete Impuls. Syst. Ser. A Math. Anal.},
  {\bf 17(4)} (2010), 555--571.

  
\bibitem{figueroa_2015} 
P. Figueroa, M. Pinto, {\it Poincar\'e's problem in the
class of almost periodic type functions},
Bulletin of the Belgiam Mathematical Society Simon Stevin,
  {\bf 22(2)}(2015), 177--198.

  
  
\bibitem{gazzola_2006}
 F. Gazzola, H.C.  Grunau; Radial entire solutions for supercritical biharmonic equations,
 {\em Math. Ann.},  {\bf 334(4)}  (2006),  905--936.  
  
\bibitem{harris_lutz_1974}
W.A. Harris Jr., D.A. Lutz;
 On the asymptotic integration of linear differential systems.
 {\em Journal of Mathematical Analysis and Applications}, {\bf 48(1)} (1974), 
 1--16.
  
\bibitem{harris_lutz_1977}
W.A. Harris Jr., D.A. Lutz;
 A unified theory of asymptotic integration.
 {\em J. Math. Anal. Appl.}, {\bf 57(3)} (1977), 571--586.
  
  
\bibitem{hartman_1948}
P. Hartman;
 Unrestricted solution fields of almost-separable differential
  equations.
 {\em Trans. Amer. Math. Soc.},  {\bf 63} (1948), 560--580. 

\bibitem{hartman_wintner_1955}
P. Hartman, A. Wintner;
 Asymptotic integrations of linear differential equations.
 {\em American Journal of Mathematics}, {\bf 77(1)} (1955),  45--86. 
 
\bibitem{jangveladze_2011}
T.A. Jangveladze, G.B. Lobjanidze;
On a nonlocal boundary value problem for a fourth-order ordinary differential equation,
{\em Differential Equations}, {\bf 47(2)} (2011), 179--186.

\bibitem{karageorgis_2011}
P. Karageorgis;
Asymptotic expansion of radial solutions for supercritical biharmonic equations,
{\em Nonlinear Differential Equations and Applications,}
{\em 19(4)} (2012),  401--415.

\bibitem{lamnii_2014} 
 A. Lamnii, O.  El-khayyari, J. Dabounou; Solving linear
 fourth order boundary value problem by using a
 hyperbolic splines of order 4, {\em Int. Electron. J. Pure Appl. Math.},
 {\bf 7(2)} (2014),  85–98. 
 
\bibitem{levinson1948}
N. Levinson;
 The asymptotic nature of solutions of linear systems of differential
  equations.
 {\em Duke Mathematical Journal}, {\bf 15(1)} (1948), 111--126.

\bibitem{olver_book}
F.W.J. Olver;
 {\em Asymptotics and Special Functions}.
 (Reprint of the 1974 original Academic Press, New York)
 AKP Classics, A.K. Peters, Ltd., Wellesley, MA, 1997.

\bibitem{perron_1909}
O. Perron;
Ber einen satz des henr Poincar\'e.
 {\em J. Reine Angew. Math.}, {\bf 136} (1909), 17--37. 
 
\bibitem{pfeiffer_1970}
G.W. Pfeiffer;
 {\em Asymptotic solutions of the equation {$y'''+qy' +ry=0$}}.
 ProQuest LLC, Ann Arbor, MI, 
 Thesis (Ph.D.)--University of Georgia, 1970.


\bibitem{pfeiffer_1972}
G.W. Pfeiffer;
 Asymptotic solutions of {$y'''+qy' +ry=0$}.
 {\em J. Differential Equations}, {\bf 11} (1972), 145--155, . 
 
\bibitem{barbara_2013}
B. Pietruczuk;
Resonance phenomenon for potentials of wigner–von neumann type.
In {\em   Geometric Methods in Physics}
(P. Kielanowski, S.T. Ali, A. Odesskii, A. Odzijewicz, 
M. Schlichenmaier, and T. Voronov, editors). 
Trends in Mathematics, Springer Basel, pages 203--207, 2013.

\bibitem{pinto_2003}
M. Pinto;
 Null solutions of difference systems under vanishing perturbation.
 {\em J. Difference Equ. Appl.}, 
 {\bf 9(1)} (2003), 1--13.

\bibitem{poincare_1885}
H. Poincar\'e;
 Sur les equations lineaires aux differentielles ordinaires et
  aux differences finies.
 {\em Amer. J. Math.}, {\bf 7(3)} (1885), 203--258.

\bibitem{simsa_1988}
J. {\v{S}}im{\v{s}}a;
 An extension of a theorem of {P}erron.
 {\em SIAM J. Math. Anal.}, {\bf 19(2)} (1988), 460--472.
 
\bibitem{stepin_2005}
S.A. Stepin;
 The wkb method and dichotomy for ordinary differential equations.
 {\em Doklady Mathematics}, {\bf 72(2)} (2005), 783--786.

\bibitem{stepin_2010}
S.A. Stepin;
 Asymptotic integration of nonoscillatory second-order differential
  equations.
 {\em Doklady Mathematics}, {\bf 82(2)} (2010), 751--754.

\bibitem{timoshenko_book} 
S.P. Timoshenko; Theory of Elastic Stability, 
McGraw-Hill Book, New York, NY, USA, 2nd edition,
1961. 


\bibitem{yang_eigenvalue} 
Y.-X. Yao, Y.-T. Shen, Z.-H. Chen;
Biharmonic Equation and an Improved Hardy Inequality,
{\em Acta Mathematicae Applicatae Sinica},
{\bf 20(3)} (2004), 433--440.



\end{thebibliography}
\end{document}